\documentclass[a4paper]{amsart}
\usepackage{a4wide}
\usepackage{amsmath,amssymb,amsthm,mathtools}
\usepackage[mathcal]{eucal}
\usepackage{xcolor}
\usepackage{url}
\usepackage[colorlinks=true,linkcolor=black,citecolor=black,urlcolor=blue]{hyperref}

\newtheorem{theorem}{Theorem}[section]
\newtheorem{lemma}[theorem]{Lemma}
\newtheorem{proposition}[theorem]{Proposition}
\newtheorem{corollary}[theorem]{Corollary}
\newtheorem*{theoremA}{Theorem A}
\newtheorem*{theoremB}{Theorem B}
\newtheorem*{theoremC}{Theorem C}
\newtheorem*{theoremD}{Theorem D}
\theoremstyle{remark}
\newtheorem{remark}[theorem]{Remark}

\DeclareMathOperator{\conv}{conv}

\newcommand{\D}{\mathbb D}
\newcommand{\T}{\mathbb T}
\newcommand{\C}{\mathbb C}
\newcommand{\N}{\mathbb N}
\newcommand{\calA}{\mathcal A}
\newcommand{\calB}{\mathcal B}
\newcommand{\calK}{\mathcal K}
\newcommand{\calQ}{\mathcal Q}
\newcommand{\calF}{\mathcal F}
\newcommand{\calS}{\mathcal S}
\newcommand{\eps}{\varepsilon}

\title[Spectral constants for numerical ranges]{Spectral constants for algebraic numerical ranges}
\author[T.~Kania]{Tomasz Kania}
\address[T.~Kania]{Mathematical Institute\\
Czech Academy of Sciences\\
\v Zitn\'a 25\\
115 67 Praha 1\\
Czech Republic
and
Institute of Mathematics and Computer Science\\
Jagiellonian University\\
{\L}ojasiewicza 6, 30-348 Krak\'{o}w, Poland
}
\email{kania@math.cas.cz, tomasz.marcin.kania@gmail.com}
\thanks{RVO: 67985840.}
\date{}

\begin{document}

\begin{abstract}
We study spectral-set constants associated with the algebraic numerical range
in unital Banach algebras.  If \(\Gamma_n\) denotes the universal constant
for algebraic elements of degree at most \(n\), then
\[
        \Gamma_1=1,
        \qquad
        2n-1\leqslant\Gamma_n<\infty\qquad(n\geqslant2).
\]
Thus the algebraic numerical range is a spectral set with a constant depending
only on the algebraic degree.  In degree two we obtain
\[
        3\leqslant\Gamma_2
        \leqslant\sqrt{1+(2\mathrm e-1)^2}<4.55.
\]
For unital \(C^*\)-algebras we prove a gap theorem: the algebra-level constant
is one precisely in the commutative case, whereas every non-commutative
algebra has constant at least two.  Moreover, the constant of the Jiang--Su
algebra is exactly the universal Crouzeix constant.  In the opposite
direction, we construct a norm-one operator with a contractive polynomial
calculus on the unit disc but infinite numerical-range spectral constant, and
we show that the canonical left shift on every spreading combinatorial space
in a broad class has infinite constant.  These results settle the three
questions posed by Blazhko, Homza, Schwenninger, de Vries and Wojtylak.
\end{abstract}

\maketitle

\section{Introduction}

Let \(\calA\) be a complex unital Banach algebra whose unit has norm one.
For \(a\in\calA\), the algebraic numerical range \(V(a,\calA)\) is obtained by
evaluating \(a\) on the states of \(\calA\).  We study the least constant in
the polynomial functional-calculus estimate
\[
        \|p(a)\|\leqslant C
        \sup_{z\in V(a,\calA)}|p(z)|.
\]
Equivalently, omitting polynomials for which the denominator vanishes, put
\[
        \Psi(a,\calA)=
        \sup_{p\in\C[z]}
        \frac{\|p(a)\|}{\sup_{z\in V(a,\calA)}|p(z)|}.
\]
The paper is organised around a rigidity--obstruction dichotomy.  Algebraic
relations and \(C^*\)-structure force finite numerical-range calculi, whereas
non-Hilbertian infinite-dimensional geometry may destroy them even in the
presence of a contractive polynomial calculus on the unit disc.

The questions originate in the work of Blazhko, Homza, Schwenninger, de Vries
and Wojtylak \cite{BHDVW}.  The results below are stated and proved as a
standalone theory, but they also settle the three questions in
\cite[Section~8]{BHDVW}.  The algebraic degree of an element means the degree
of its minimal polynomial.  For \(n\geqslant1\), define
\[
       \Gamma_n=
       \sup \{\Psi(a,\calA):\; \calA \text{ is a unital Banach algebra and }
       a\in\calA \text{ has algebraic degree at most }n\}.
\]
Our main positive result is the following.

\begin{theoremA}
One has
\[
       \Gamma_1=1,
       \qquad
       2n-1\leqslant\Gamma_n<\infty\quad(n\geqslant2).
\]
In particular, every fixed algebraic degree admits a universal
numerical-range spectral constant, and the same is true for every finite
matrix algebra equipped with an arbitrary unital submultiplicative norm.
\end{theoremA}

The degree-two argument is considerably sharper.

\begin{theoremB}
One has
\[
       3\leqslant \Gamma_2
       \leqslant\sqrt{1+(2\mathrm e-1)^2}<4.55.
\]
\end{theoremB}

At the algebra level, \(C^*\)-structure yields a different rigidity
phenomenon.  Let \(C_{\mathrm{Cr}}\) denote the universal Crouzeix constant.

\begin{theoremC}
For every unital \(C^*\)-algebra \(\calA\),
\[
       \Psi_{\calA}<2
       \quad\Longleftrightarrow\quad
       \Psi_{\calA}=1
       \quad\Longleftrightarrow\quad
       \calA\text{ is commutative}.
\]
If \(\calA\) is non-commutative, then
\[
       2\leqslant\Psi_{\calA}\leqslant C_{\mathrm{Cr}}
       \leqslant1+\sqrt2.
\]
For the Jiang--Su algebra \(\mathcal Z\), one has the exact identification
\[
       \Psi_{\mathcal Z}=C_{\mathrm{Cr}}.
\]
\end{theoremC}

The corresponding negative phenomena are already visible in constructions built from canonical shifts.

\begin{theoremD}
\begin{enumerate}
\item There are a Banach space \(X\) and an operator \(T\in\calB(X)\) with
      \(\|T\|=1\) such that
      \[
             \|p(T)\|\leqslant\sup_{|z|\leqslant1}|p(z)|
             \qquad(p\in\C[z]),
      \]
      but \(\Psi(T,\calB(X))=\infty\).
\item For every spreading family of finite subsets of \(\N\) which covers
      \(\N\) and contains sets of arbitrarily large cardinality, the canonical
      left shift on the associated combinatorial space is a contraction and
      has infinite numerical-range spectral constant.
\end{enumerate}
\end{theoremD}

The proof of Theorem~A has two complementary parts.  In the large-norm regime,
a ratio-drop argument forces a disc of radius comparable to the operator norm
inside the algebraic numerical range; Cauchy's divided-difference formula then
controls the Newton--Hermite functional calculus.  In the bounded-norm regime,
the roots of the minimal polynomial are divided into uniformly separated
clusters.  A quantitatively controlled B\'ezout projection reduces the
algebraic degree and closes an induction.  The lower bound is produced by
scale-separated upper bidiagonal matrices on \(\ell_1^n\), which convert
almost alternating Schur interpolation data into a last column of asymptotic
norm \(2n-1\).

Theorem~C provides the conceptual endpoint of the positive part of the paper.  Every non-commutative \(C^*\)-algebra contains a non-zero
square-zero element, which forces the lower obstruction two.  The Jiang--Su
algebra, introduced in \cite{JiangSu}, is a natural single-algebra test because
its prime dimension-drop subalgebras detect matrix Crouzeix constants of
unbounded size; the required embeddings are supplied by
\cite[Theorem~2.2 and Proposition~3.3]{RordamWinter}.

Theorem~D sharpens the two relevant constructions in \cite{BHDVW}.  Their
Example~6.2 gives essentially the same direct-sum operator with polynomial
bound \(2\sqrt3/3\); the interpolation argument here lowers that bound to the
optimal value one.  Their Theorem~6.3 proves
\(\Psi_{\calB(\calS)}=\infty\) by using a family of cut shifts.  Here the
single canonical left shift itself has infinite constant, which answers the
individual-operator question.

Section~2 records the numerical-range, interpolation and divided-difference
facts used throughout.  Section~3 proves Theorems~A and~B, Section~4 proves
Theorem~C, and Section~5 contains the two counterexamples in Theorem~D.
Further shift consequences are collected in Section~6.  The explicit cubic
bookkeeping and the bounded-order root-of-unity estimate are placed in the
appendix, since they are quantitatively useful but not needed for the main
conceptual narrative.

\section{Preliminaries}

Throughout, Banach algebras are complex and unital, and the unit has norm one.  We use the duality convention
\[
       \langle x,x^*\rangle=x^*(x)\qquad(x\in X,\ x^*\in X^*)
\]
for Banach spaces and their duals; the same notation is used for elements of a Banach algebra and functionals on it.  For \(a\in\calA\) we write
\[
 V(a,\calA)=\{\langle a,\varphi\rangle:\varphi\in\calA',\ \|\varphi\|=1=\langle 1,\varphi\rangle\}
\]
for the algebraic numerical range and
\[
 \Psi(a,\calA)=
 \sup_{\substack{p\in\C[z]\\ \sup_{z\in V(a,\calA)}|p(z)|>0}}
 \frac{\|p(a)\|}{\sup_{z\in V(a,\calA)}|p(z)|}.
\]
Thus scalar elements have constant one.  For a unital Banach algebra put
\[
       \Psi_{\calA}=\sup_{a\in\calA}\Psi(a,\calA).
\]
The standard facts about algebraic numerical ranges used below may be found in
\cite[Chapter~1]{BonsallDuncan}; see also \cite[Section~2]{BHDVW}.  In
particular, \(V(a,\calA)\) is compact and convex and contains the spectrum of
\(a\).  If \(b=\alpha a+\beta1\) with \(\alpha\neq0\), then
\(V(b)=\alpha V(a)+\beta\) and \(\Psi(b)=\Psi(a)\).  We shall also use the
invariance under unital subalgebras: if \(\calB\subseteq\calA\) has the same
unit, then
\[
       V(b,\calB)=V(b,\calA)\qquad(b\in\calB).
\]
Indeed, states restrict to states, and states extend by Hahn--Banach.  We shall
repeatedly use the support-function formula
\begin{equation}\label{eq:support}
       \sup_{w\in V(a)}\operatorname{Re}(\zeta w)=\lim_{t\downarrow0}\frac{\|1+t\zeta a\|-1}{t}\qquad(\zeta\in\T).
\end{equation}
If \(a\in A\) and \(b\in B\), where \(A\oplus_\infty B\) is the unital
\(\ell_\infty\)-direct sum, then
\begin{equation}\label{eq:diagonal-sum}
       V(a\oplus b,A\oplus_\infty B)
       =\conv\bigl(V(a,A)\cup V(b,B)\bigr).
\end{equation}
Indeed,
\[
 \|1+t\zeta(a\oplus b)\|
 =\max\{\|1+t\zeta a\|,\|1+t\zeta b\|\},
\]
so \eqref{eq:support} identifies the support function on the left with the
maximum of the two support functions on the right.  This is the support
function of the indicated convex hull.  Compare the same direct-sum
calculation in \cite[Example~6.2]{BHDVW}.
\begin{lemma}\label{lem:support-containment}
Let \(K,L\subseteq\C\) be compact and convex.  If
\[
       \sup_{w\in K}\operatorname{Re}(\zeta w)
       \geqslant
       \sup_{w\in L}\operatorname{Re}(\zeta w)
       \qquad(\zeta\in\T),
\]
then \(L\subseteq K\).  In particular, if
\(\sup_{w\in K}\operatorname{Re}(\zeta w)\geqslant r\) for every \(\zeta\in\T\),
then \(r\overline\D\subseteq K\).
\end{lemma}

\begin{proof}
If a point of \(L\) did not belong to \(K\), the strict separation theorem in
\(\mathbb R^2\cong\C\) would produce a direction in which the support
function of \(L\) is larger than that of \(K\).  The final assertion follows
by taking \(L=r\overline\D\).
\end{proof}

\begin{lemma}\label{lem:left-regular}
Let \(\calA\) be a unital Banach algebra and, for \(a\in\calA\), let
\(L_a\in\calB(\calA)\) denote left multiplication by \(a\).  The map
\[
       \lambda:\calA\longrightarrow\calB(\calA),
       \qquad \lambda(a)=L_a,
\]
is a unital isometric algebra embedding.  Moreover,
\[
       V(L_a,\calB(\calA))=V(a,\calA),
       \qquad
       \Psi(L_a,\calB(\calA))=\Psi(a,\calA),
\]
and \(a\) and \(L_a\) have the same minimal polynomial.
\end{lemma}

\begin{proof}
We have \(\|L_a\|\leqslant\|a\|\), whereas
\(\|L_a\|\geqslant\|L_a1\|=\|a\|\).  Also
\[
       p(L_a)=L_{p(a)}\qquad(p\in\C[z]),
\]
and \(L_b=0\) if and only if \(b=0\); this proves the assertions about
minimal polynomials and polynomial norms.  Finally, \(\lambda(\calA)\) is a
unital subalgebra of \(\calB(\calA)\), so subalgebra invariance gives
\[
       V(L_a,\calB(\calA))=V(L_a,\lambda(\calA))=V(a,\calA).
\]
The equality of the spectral constants follows immediately.
\end{proof}

\begin{lemma}\label{lem:invariant-restriction}
Let \(T\in\calB(X)\), and let \(Y\subseteq X\) be a non-zero closed \(T\)-invariant subspace.  Then
\[
       V(T|_Y,\calB(Y))\subseteq V(T,\calB(X)).
\]
\end{lemma}

\begin{proof}
For every \(\zeta\in\T\) and \(t>0\),
\[
       \|I_Y+t\zeta T|_Y\|\leqslant\|I_X+t\zeta T\|.
\]
Taking right derivatives at \(t=0\) in the support-function formula \eqref{eq:support} shows that the support function of \(V(T|_Y)\) is bounded above by that of \(V(T)\) in every direction.  Lemma~\ref{lem:support-containment} now gives the asserted inclusion.
\end{proof}

We recall the divided-difference notation used below.  Given an ordered list of
nodes \(\boldsymbol\lambda=(\lambda_0,\ldots,\lambda_{m-1})\), with
repetitions allowed, put
\[
       N_0(z)=1,
       \qquad
       N_j(z)=\prod_{k=0}^{j-1}(z-\lambda_k)
       \quad(1\leqslant j\leqslant m-1).
\]
The Newton--Hermite interpolant of a function \(f\) at these nodes is
\begin{equation}\label{eq:newton-hermite-form}
       H_{\boldsymbol\lambda}f(z)
       =\sum_{j=0}^{m-1}
       f[\lambda_0,\ldots,\lambda_j]N_j(z).
\end{equation}
It is the unique polynomial of degree less than \(m\) which matches the
values and derivatives of \(f\) prescribed by the multiplicities of the
nodes; see \cite[Proposition~7 and formulas~(5), (6)]{deBoor}.  In
particular, formula~(14) of \cite{deBoor} gives
\(f[\lambda,\ldots,\lambda]=f^{(j)}(\lambda)/j!\) when \(\lambda\) is
repeated \(j+1\) times.  For pairwise distinct nodes one has the Lagrange
representation
\begin{equation}\label{eq:lagrange-divided-difference}
       f[\lambda_0,\ldots,\lambda_j]
       =\sum_{k=0}^j
       \frac{f(\lambda_k)}
       {\displaystyle\prod_{\substack{0\leqslant \ell\leqslant j\\ \ell\neq k}}
       (\lambda_k-\lambda_\ell)};
\end{equation}
see \cite[Example~9, formula~(41)]{deBoor}.  Finally, if \(f\) is
holomorphic on and inside a positively oriented contour \(\Gamma\) which
surrounds the nodes, then Cauchy's divided-difference formula is
\begin{equation}\label{eq:cauchy-divided-difference}
       f[\lambda_0,\ldots,\lambda_j]
       =\frac{1}{2\pi i}\int_\Gamma
       \frac{f(z)}{\prod_{k=0}^j(z-\lambda_k)}\,dz.
\end{equation}
This is formula~(51) in \cite{deBoor}; repeated nodes are included by
confluence.  We refer to \cite[Chapter~1]{Higham} for the use of these
formulae in polynomial and holomorphic matrix functional calculus.

A Littlewood polynomial of length \(n\) is a polynomial
\(q(z)=\sum_{k=0}^{n-1}\eps_kz^k\) with \(\eps_k\in\{-1,1\}\).  Let
\((\eps_k)_{k\geqslant0}\) be the Rudin--Shapiro sequence and set
\[
       q_n(z)=\sum_{k=0}^{n-1}\eps_kz^k.
\]
Balister's estimate \cite[Theorem~1]{Balister} gives
\[
       \sup_{|z|=1}|q_n(z)|
       \leqslant \sqrt{6n-2}-1<\sqrt{6n}.
\]
By the maximum-modulus principle, the same estimate holds on
\(\overline\D\).  Thus we fix these polynomials \(q_n\) throughout and take
\(\Delta=\sqrt6\).  They are the Littlewood test polynomials used in
Sections~\ref{sec:polynomially-bounded} and~\ref{sec:shift-obstructions}; no
other property of the Rudin--Shapiro construction is needed.

\section{Algebraic elements of fixed degree}

\begin{proposition}\label{prop:gamma-n-lower}
For every integer \(n\geqslant2\),
\[
       \Gamma_n\geqslant 2n-1.
\]
In particular, \(\Gamma_3\geqslant5\).
\end{proposition}

\begin{proof}
For \(n=2\), let
\[
       P=\begin{pmatrix}1&1\\0&0\end{pmatrix}
\]
act on \(\ell_1^2\).  Then \(P^2=P\), and the column Gershgorin description of the algebraic numerical range on \(\ell_1^m\) \cite[Theorem~7.2]{BHDVW} gives
\[
       V(P,\calB(\ell_1^2))=\conv(\{1\}\cup\overline\D)=\overline\D.
\]
For \(0<\rho<1\), put \(g_\rho(z)=(z-\rho)/(1-\rho z)\).  Then \(\|g_\rho\|_{\infty,\overline\D}=1\), \(g_\rho(1)=1\), and \(g_\rho(0)=-\rho\).  Since \(g_\rho\) is holomorphic in a neighbourhood of \(\overline\D\), choose polynomials \(p_\nu\) converging uniformly to \(g_\rho\) on \(\overline\D\).  The identity
\[
       p_\nu(P)=p_\nu(0)(I-P)+p_\nu(1)P
\]
shows that \(p_\nu(P)\to g_\rho(P)\) in operator norm.  Moreover
\[
       g_\rho(P)=\begin{pmatrix}1&1+\rho\\0&-\rho\end{pmatrix},
\]
whose \(\ell_1^2\)-operator norm is \(1+2\rho\).  Letting \(\rho\uparrow1\) gives \(\Gamma_2\geqslant3\).

Assume now that \(n\geqslant3\).  Let \(0<\delta<1\).  Put
\[
       \eps_1=0,
       \qquad
       \eps_j=\delta^{\,n-j}\quad(2\leqslant j\leqslant n-1),
       \qquad
       \eps_n=1,
       \qquad
       \lambda_j=1-\eps_j.
\]
Let \(A_\delta\in\calB(\ell_1^n)\) be the upper bidiagonal matrix with diagonal entries \(\lambda_1,\ldots,\lambda_n\) and superdiagonal entries \(\eps_2,\ldots,\eps_n\).  The numbers \(\lambda_j\) are distinct, so the minimal polynomial of \(A_\delta\) has degree \(n\).  The column Gershgorin discs are
\[
       \{1\},\quad \lambda_j+\eps_j\overline\D\quad(2\leqslant j\leqslant n).
\]
Each of these discs is contained in \(\overline\D\), and the last one is \(\overline\D\) itself.  Hence \(V(A_\delta,\calB(\ell_1^n))=\overline\D\).

Choose numbers \(\eta_j\) by
\[
       \eta_1=\eps_2^2,
       \qquad
       \eta_j=(\eps_j\eps_{j+1})^{1/2}\quad(2\leqslant j\leqslant n-1),
\]
and set \(r_j=1-\eta_j\).  With \(b_r(z)=(z-r)/(1-rz)\), put
\[
       q_\delta(z)=\prod_{j=1}^{n-1}b_{r_j}(z).
\]
This is a finite Blaschke product, so \(\|q_\delta\|_{\infty,\overline\D}=1\).  Since
\[
       b_{1-\eta}(1-\eps)=\frac{\eta-\eps}{\eta+\eps-\eta\eps},
\]
the sign limit can be read off factor by factor.  Fix \(i\).  From the
choice of the intermediate scales \(\eta_j\), including
\(\eta_1=\eps_2^2\), one has
\[
       \frac{\eta_j}{\eps_i}\longrightarrow0\quad(j<i),
       \qquad
       \frac{\eps_i}{\eta_j}\longrightarrow0\quad(j\geqslant i),
\]
where the second statement is immediate for \(i=1\) because \(\eps_1=0\).
Consequently,
\[
       b_{1-\eta_j}(1-\eps_i)\longrightarrow
       \begin{cases}
       -1,&j<i,\\
        1,&j\geqslant i.
       \end{cases}
\]
There are exactly \(i-1\) negative factors, and therefore
\[
       q_\delta(\lambda_i)\longrightarrow \sigma_i:=(-1)^{i-1}
       \qquad(1\leqslant i\leqslant n)
\]
as \(\delta\downarrow0\).

For a function \(q\) holomorphic in a neighbourhood of \(\overline\D\), the standard divided-difference formula for functions of an upper bidiagonal matrix \cite[Chapter~1]{Higham} gives
\[
       (q(A_\delta))_{ij}=\eps_{i+1}\cdots\eps_j\,
       q[\lambda_i,\ldots,\lambda_j]
       \quad(i<j),
       \qquad
       (q(A_\delta))_{ii}=q(\lambda_i).
\]
Let \(F_i^\delta\) denote the \(i\)-th entry of the last column of
\(q_\delta(A_\delta)\).  Combining the preceding bidiagonal-matrix formula
with \eqref{eq:lagrange-divided-difference} gives
\[
       F_i^\delta=
       \sum_{k=i}^n q_\delta(\lambda_k)c_{ik}^{(\delta)},
       \qquad
       c_{ik}^{(\delta)}=
       \frac{\eps_{i+1}\cdots\eps_n}
       {\displaystyle\prod_{\substack{\ell=i\\ \ell\neq k}}^n
       (\lambda_k-\lambda_\ell)}
       \quad(i<n),
\]
whereas \(F_n^\delta=q_\delta(\lambda_n)\).  Since
\(\lambda_k-\lambda_\ell=\eps_\ell-\eps_k\), the first two coefficients
satisfy
\[
       c_{ii}^{(\delta)}
       =\prod_{\ell=i+1}^n
       \frac{\eps_\ell}{\eps_\ell-\eps_i}
       \longrightarrow1
\]
and
\[
       c_{i,i+1}^{(\delta)}
       =\frac{\eps_{i+1}}{\eps_i-\eps_{i+1}}
       \prod_{\ell=i+2}^n
       \frac{\eps_\ell}{\eps_\ell-\eps_{i+1}}
       \longrightarrow-1.
\]
For \(k\geqslant i+2\), we make the scale estimate explicit.  Assume
\(0<\delta\leqslant1/2\).  If \(\ell<k\), then
\(\eps_\ell/\eps_k\leqslant\delta\), and hence
\[
       |\eps_\ell-\eps_k|\geqslant(1-\delta)\eps_k
       \geqslant\frac12\eps_k.
\]
If \(\ell>k\), then \(\eps_k/\eps_\ell\leqslant\delta\), and similarly
\[
       |\eps_\ell-\eps_k|\geqslant\frac12\eps_\ell.
\]
There are \(k-i\) factors of the first type.  Cancelling the factors with
\(\ell>k\) against the numerator gives
\[
\begin{split}
       |c_{ik}^{(\delta)}|
       &\leqslant
       2^{n-i}
       \frac{\eps_{i+1}\cdots\eps_{k-1}}
       {\eps_k^{\,k-i-1}}\\
       &=2^{n-i}\,
       \delta^{\sum_{j=i+1}^{k-1}(k-j)}
       =2^{n-i}\,
       \delta^{(k-i-1)(k-i)/2}.
\end{split}
\]
Thus, for fixed \(n\), one may take \(C_n=2^n\) in the uniform estimate
\[
       |c_{ik}^{(\delta)}|
       \leqslant C_n\delta^{(k-i-1)(k-i)/2},
       \qquad k\geqslant i+2,
\]
and in particular \(c_{ik}^{(\delta)}\to0\).
It follows that
\[
       F_i^\delta\longrightarrow \sigma_i-\sigma_{i+1}=2(-1)^{i-1}
       \quad(i<n),
       \qquad
       F_n^\delta\longrightarrow\sigma_n.
\]
Consequently,
\[
\begin{split}
       \|q_\delta(A_\delta)\|_{\ell_1^n\to\ell_1^n}
       &\geqslant \|q_\delta(A_\delta)e_n\|_1\\
       &=\sum_{i=1}^n|F_i^\delta|
       \longrightarrow 2(n-1)+1=2n-1.
\end{split}
\]
Finally, \(A_\delta\) has distinct eigenvalues and is therefore
 diagonalizable.  If
\(A_\delta=S_\delta\operatorname{diag}(\lambda_1,\ldots,\lambda_n)
S_\delta^{-1}\) and polynomials \(p_\nu\) converge uniformly to
\(q_\delta\) on \(\overline\D\), then
\[
       p_\nu(A_\delta)\longrightarrow q_\delta(A_\delta)
\]
in operator norm and
\(\|p_\nu\|_{\infty,\overline\D}\to
\|q_\delta\|_{\infty,\overline\D}=1\).  Hence, for every fixed
\(\delta\),
\[
       \Psi(A_\delta,\calB(\ell_1^n))
       \geqslant
       \lim_{\nu\to\infty}
       \frac{\|p_\nu(A_\delta)\|}
       {\|p_\nu\|_{\infty,\overline\D}}
       =\|q_\delta(A_\delta)\|.
\]
Letting \(\delta\downarrow0\) proves the assertion.
\end{proof}

\begin{remark}\label{rem:exact-degree-two-example}
For the idempotent
\[
       P=\begin{pmatrix}1&1\\0&0\end{pmatrix}
       \in\calB(\ell_1^2)
\]
one has
\[
       V(P,\calB(\ell_1^2))=\overline\D
       \qquad\text{and}\qquad
       \Psi(P,\calB(\ell_1^2))=3.
\]
Indeed, since \(P^2=P\),
\[
       p(P)=
       \begin{pmatrix}
       p(1)&p(1)-p(0)\\
       0&p(0)
       \end{pmatrix}.
\]
Consequently,
\[
\begin{split}
       \|p(P)\|_{\ell_1^2\to\ell_1^2}
       &=\max\bigl\{|p(1)|,\ |p(1)-p(0)|+|p(0)|\bigr\}\\
       &\leqslant3\sup_{z\in\overline\D}|p(z)|.
\end{split}
\]
Thus \(\Psi(P,\calB(\ell_1^2))\leqslant3\).  Conversely, the functions
\[
       g_\rho(z)=\frac{z-\rho}{1-\rho z}
       \qquad(0<\rho<1)
\]
satisfy
\[
       \|g_\rho\|_{\infty,\overline\D}=1,
       \qquad
       g_\rho(1)=1,
       \qquad
       g_\rho(0)=-\rho,
\]
and hence
\[
       \|g_\rho(P)\|_{\ell_1^2\to\ell_1^2}=1+2\rho.
\]
Since each \(g_\rho\) is holomorphic in a neighbourhood of
\(\overline\D\), it can be approximated there uniformly by polynomials.
Letting \(\rho\uparrow1\) proves the reverse inequality.
\end{remark}

\begin{remark}[Why the lower-bound matrices are not ordinary shifts]
It may be useful to explain the form of the matrices \(A_\delta\) in
Proposition~\ref{prop:gamma-n-lower}.  The first natural attempt would be to
use the nilpotent shift \(N_n\) on \(\ell_1^n\), since the column Gershgorin
formula gives
\[
       V(N_n,\calB(\ell_1^n))=\overline\D.
\]
However, this cannot yield the linear lower bound \(2n-1\).  Indeed, if
\[
       q(z)=\sum_{k\geqslant0}a_kz^k,
       \qquad
       \|q\|_{\infty,\D}\leqslant1,
\]
then \(q\in H^2\) and
\[
       \sum_{k\geqslant0}|a_k|^2\leqslant1.
\]
Since \(N_n^n=0\),
\[
       \|q(N_n)\|_{\ell_1^n\to\ell_1^n}
       \leqslant\sum_{k=0}^{n-1}|a_k|
       \leqslant\sqrt n.
\]
Thus a plain shift, even when tested on finite Blaschke products, detects only
Taylor coefficients of a Schur function and produces at most a
\(\sqrt n\)-effect.

The matrices \(A_\delta\) are designed to exploit interpolation instead.
Their first \(n-1\) diagonal points
\(\lambda_j=1-\eps_j\) cluster near the boundary point \(1\) on widely
separated scales, while the final point is \(\lambda_n=0\).  The
superdiagonal entry in the \(j\)-th column is \(\eps_j\), so the corresponding
column Gershgorin disc is
\[
       \lambda_j+\eps_j\overline\D
       =1-\eps_j+\eps_j\overline\D
       \subseteq\overline\D,
\]
and the final disc is the whole unit disc.  Thus
\(V(A_\delta,\calB(\ell_1^n))=\overline\D\), while the upper triangular
structure still permits large values of \(q(A_\delta)\).

The zeros of the Blaschke product \(q_\delta\) lie at intermediate scales
between neighbouring spectral points.  Consequently,
\[
       q_\delta(\lambda_i)\longrightarrow(-1)^{i-1},
\]
so these values oscillate almost maximally along the ordered spectrum.  The
functional calculus for an upper bidiagonal matrix expresses the last column
of \(q_\delta(A_\delta)\) through divided differences.  The product
\[
       \eps_{i+1}\cdots\eps_n
\]
of superdiagonal entries is precisely the normalising factor which cancels
the dominant denominators.  Scale separation makes every interpolation node
except the two nearest ones asymptotically negligible.  More precisely,
\[
       F_i^\delta-
       \bigl(q_\delta(\lambda_i)-q_\delta(\lambda_{i+1})\bigr)
       \longrightarrow0,
\]
and hence
\[
       F_i^\delta\longrightarrow2(-1)^{i-1}
       \qquad(i<n).
\]
The last entry is \(F_n^\delta=q_\delta(\lambda_n)\) and tends to
\((-1)^{n-1}\).  The \(\ell_1\)-norm of the last column therefore tends to
\[
       2+\cdots+2+1=2n-1.
\]
In this sense, \(A_\delta\) is a finite-difference amplifier: its numerical
range remains the unit disc, but it converts almost alternating interpolation
values of a Schur function into a large column sum.
\end{remark}

\begin{theorem}\label{thm:degree-two}
For the degree-two constant one has
\[
       3\leqslant \Gamma_2
       \leqslant \sqrt{1+(2\mathrm e-1)^2}<4.55.
\]
\end{theorem}

We first prove two geometric estimates for the numerical ranges of square-zero elements and involutions.  By Lemma~\ref{lem:left-regular}, it is enough throughout to work with operators on Banach spaces.

\begin{lemma}\label{lem:nilpotent-fat}
There is an absolute constant $\kappa>0$ such that $\kappa\|T\|\overline\D\subseteq V(T)$ whenever $T^2=0$.  More precisely, one may take
\[
       \kappa=\sup_{R>2}\frac{\log(R-1)}{R}
       =0.2784645\ldots .
\]
\end{lemma}

\begin{proof}
It is enough to prove the assertion for $T\in\calB(X)$.  The case $T=0$ is trivial, so put $M=\|T\|>0$ and fix $\zeta\in\T$.  Choose $x\in X$ with $\|x\|=1$ and $m=\|Tx\|>(1-\eps)M$, and set $y=Tx/m$.  Then $\|y\|=1$ and $Ty=0$.  The vectors $x$ and $y$ are linearly independent: otherwise $Tx$ would be a scalar multiple of $x$, and $T^2x=0$ would force $Tx=0$, contrary to $m>0$.  For $r\geqslant0$ let $f(r)=\|x+r\zeta y\|$.  Thus $f$ is positive, $f(0)=1$, and, for $R>2$, $f(R)\geqslant R-1$.

Since $(I+t\zeta T)(x+r\zeta y)=x+(r+tm)\zeta y$, we have, for $t>0$,
\[
       \|I+t\zeta T\|\geqslant
       \frac{\|(I+t\zeta T)(x+r\zeta y)\|}{\|x+r\zeta y\|}
       =\frac{f(r+tm)}{f(r)}.
\]
Subtracting $1$, dividing by $t$, and letting $t\downarrow0$, formula \eqref{eq:support} gives
\[
       \sup_{w\in V(T)}\operatorname{Re}(\zeta w)\geqslant m\,\frac{f'_+(r)}{f(r)}.
\]
The positive convex function $f$ is locally absolutely continuous and its right derivative agrees almost everywhere with the derivative.  Hence
\[
       \int_0^R \frac{f'_+(r)}{f(r)}\,dr=\log f(R)-\log f(0)\geqslant\log(R-1).
\]
Therefore some $r\in[0,R]$ satisfies $f'_+(r)/f(r)\geqslant\log(R-1)/R$.  Letting $\eps\downarrow0$ and then taking the supremum over $R>2$ gives $\sup_{w\in V(T)}\operatorname{Re}(\zeta w)\geqslant\kappa M$.  Lemma~\ref{lem:support-containment} gives $\kappa M\overline\D\subseteq V(T)$.
\end{proof}

\begin{lemma}\label{lem:involution-fat}
Let $U^2=1$ and put $M=\|U\|$.  If $M>2\mathrm e-1$, then
\[
       \frac{M}{2\mathrm e-1}\,\overline\D\subseteq V(U).
\]
\end{lemma}

\begin{proof}
Put $M_0=2\mathrm e-1$.  Again it is enough to work in $\calB(X)$.  Fix $\zeta\in\T$ and write
\[
       h(\zeta)=\sup_{w\in V(U)}\operatorname{Re}(\zeta w).
\]
We shall prove $h(\zeta)\geqslant M/M_0$.  Suppose, to the contrary, that $h(\zeta)<M/M_0$.  Choose $\eps>0$ so small that
\[
       h(\zeta)<\frac{(1-\eps)M}{M_0}
       \qquad\text{and}\qquad
       (1-\eps)M>M_0,
\]
and then choose $x\in X$ with $\|x\|=1$ and
$m=\|Ux\|>(1-\eps)M$.  Put $y=Ux/m$.  Then $\|y\|=1$ and $Uy=x/m$.  Since
$m>2\mathrm e-1>1$, the vectors $x$ and $y$ are linearly independent:
otherwise $Ux=my$ would make $x$ an eigenvector of $U$ with an eigenvalue
of modulus $m$, contradicting $U^2=I$.  Consequently
$\|x+r\zeta y\|>0$ for every $r\geqslant0$.  For $r\geqslant0$ set
$f(r)=\|x+r\zeta y\|$.

Since
\[
       (I+t\zeta U)(x+r\zeta y)=x+(r+tm)\zeta y+\frac{tr\zeta^2}{m}x,
\]
the triangle inequality and \eqref{eq:support} give
\[
       h(\zeta)\geqslant m\frac{f'_+(r)}{f(r)}-\frac{r}{mf(r)}\qquad(r\geqslant0).
\]
By the choice of $x$, this implies the differential inequality
\[
       f'_+(r)<\frac1{M_0}f(r)+\frac r{m^2}\qquad(0\leqslant r\leqslant M_0).
\]
Multiplying by $\exp(-r/M_0)$ and integrating from $0$ to $M_0$, we obtain
\[
       f(M_0)<\mathrm e+\frac{\mathrm e M_0^2}{m^2}\left(1-\frac2{\mathrm e}\right).
\]
Since $m>M_0$, the right-hand side is strictly smaller than
\[
       \mathrm e+\mathrm e\left(1-\frac2{\mathrm e}\right)=2\mathrm e-2=M_0-1.
\]
This contradicts $f(M_0)\geqslant M_0-1$.  Hence $h(\zeta)\geqslant M/M_0$ for every $\zeta\in\T$, and Lemma~\ref{lem:support-containment} gives the asserted disc inclusion.
\end{proof}

\begin{lemma}\label{lem:two-point-schwarz-pick}
Let $0<a<1$ and let $q$ be holomorphic from $\D$ into $\overline\D$.  If
\[
       \alpha=\frac{q(a)+q(-a)}2,
       \qquad
       \beta=\frac{q(a)-q(-a)}2,
\]
then
\[
       |\alpha|^2+a^{-2}|\beta|^2\leqslant1.
\]
\end{lemma}

\begin{proof}
Put $u=q(a)=\alpha+\beta$ and $v=q(-a)=\alpha-\beta$.  By the Nevanlinna--Pick theorem, in its Pick-matrix form \cite[Chapter~2]{AM} (see also Pick's original paper \cite{Pick}), the Pick matrix
\[
       \begin{pmatrix}
       \dfrac{1-|u|^2}{1-a^2} & \dfrac{1-u\overline v}{1+a^2}\\[1.2ex]
       \dfrac{1-v\overline u}{1+a^2} & \dfrac{1-|v|^2}{1-a^2}
       \end{pmatrix}
\]
is positive semidefinite.  Testing it on the vector $(1,-1)$ gives
\[
       \frac{2(1-|\alpha|^2-|\beta|^2)}{1-a^2}
       -\frac{2(1-|\alpha|^2+|\beta|^2)}{1+a^2}\geqslant0,
\]
which is equivalent to $|\beta|^2\leqslant a^2(1-|\alpha|^2)$.
\end{proof}

\begin{proof}[Proof of Theorem~\ref{thm:degree-two}]
The lower bound is Proposition~\ref{prop:gamma-n-lower} with $n=2$.  We now prove the upper bound.  Set $M_0=2\mathrm e-1$ and $C_2=\sqrt{1+M_0^2}$.  Let $a$ be algebraic of degree at most two in a unital Banach algebra $\calA$.  If $a$ is scalar there is nothing to prove.  If the minimal polynomial of $a$ has a double root, then, after an affine change, $T^2=0$.  Let $K=\sup_{V(T)}|p|>0$ and put $M=\|T\|$.  The case $T=0$ is trivial.  Lemma \ref{lem:nilpotent-fat} gives $\kappa M\overline\D\subseteq V(T)$, and Schwarz--Pick applied to $q(z)=p(\kappa Mz)/K$ yields
\[
       |p'(0)|\leqslant \frac{K}{\kappa M}\left(1-\frac{|p(0)|^2}{K^2}\right).
\]
Writing $s=|p(0)|/K$ and using $p(T)=p(0)1+p'(0)T$, we get
\[
       \frac{\|p(T)\|}{K}\leqslant s+\frac1\kappa(1-s^2)
       \leqslant \frac1\kappa+\frac\kappa4.
\]
Since $\kappa<1$ and $\kappa\geqslant(\log4)/5>1/4$, while $t\mapsto t^{-1}+t/4$ is decreasing on $(0,1)$, this is at most $65/16$, which is smaller than $C_2$.

It remains to consider the case where the minimal polynomial has two distinct roots.  After an affine change, $U^2=1$.  Write $M=\|U\|$ and $K=\sup_{V(U)}|p|>0$.  As usual,
\[
       p(U)=\alpha1+\beta U,
       \qquad
       \alpha=\frac{p(1)+p(-1)}2,
       \quad
       \beta=\frac{p(1)-p(-1)}2.
\]
If $M\leqslant M_0$, then $\{ -1,1\}\subseteq\sigma(U)\subseteq V(U)$, and therefore $|p(1)|,|p(-1)|\leqslant K$.  Hence $|\alpha|^2+|\beta|^2\leqslant K^2$, and Cauchy's inequality in $\C^2$ gives
\[
       \|p(U)\|\leqslant |\alpha|+M|\beta|
       \leqslant K\sqrt{1+M^2}\leqslant C_2K.
\]
Assume now that $M>M_0$.  Lemma \ref{lem:involution-fat} gives $r\overline\D\subseteq V(U)$, where $r=M/M_0>1$.  Apply Lemma \ref{lem:two-point-schwarz-pick} to $q(z)=p(rz)/K$ and $a=1/r=M_0/M$.  Since $q(a)=p(1)/K$ and $q(-a)=p(-1)/K$, we obtain
\[
       \frac{|\alpha|^2}{K^2}+r^2\frac{|\beta|^2}{K^2}\leqslant1.
\]
Consequently
\[
       \|p(U)\|\leqslant |\alpha|+M|\beta|
       \leqslant K\sqrt{1+\frac{M^2}{r^2}}
       =C_2K.
\]
Affine invariance of $\Psi$ finishes the proof.
\end{proof}

The same ratio-drop idea gives a particularly explicit estimate for bounded-order nilpotents.

\begin{proposition}\label{prop:nilpotents}
Let \(m\geqslant2\), and put
\[
       \eta=\frac{\log2}{5},
       \qquad
       \gamma=\frac{\eta}{4},
       \qquad
       \delta_m=\frac{\eta}{2}\gamma^{m-2},
       \qquad
       N_m=\sum_{j=0}^{m-1}\delta_m^{-j}.
\]
If \(T^m=0\) in a unital Banach algebra \(\calA\), then
\[
       \Psi(T,\calA)\leqslant N_m.
\]
\end{proposition}

\begin{proof}
As before we work in $\calB(X)$.  It is enough to prove that $V(T)$ contains a disc of radius $\delta_m\|T\|$, where $\delta_m>0$ depends only on $m$, for Cauchy's estimate then gives
\[
       \|p(T)\|\leqslant \sum_{j=0}^{m-1} \delta_m^{-j}\sup_{V(T)}|p| .
\]
The case $T=0$ is trivial.  Put $M=\|T\|>0$, choose $x_0$ with $\|x_0\|=1$ and $\|Tx_0\|>(1-\eps)M$, and set $x_j=T^jx_0$.  Let $s\leqslant m$ be the first index such that $x_s=0$.  Then $x_j\neq0$ for $j<s$, and we put
\[
       R_j=\frac{\|x_{j+1}\|}{\|x_j\|}\qquad(0\leqslant j\leqslant s-1);
\]
thus $R_{s-1}=0$ and $R_0>(1-\eps)M$.  Retain the constants $\eta$ and $\gamma$ from the statement.  Since $R_0>0$ and $R_{s-1}=0$, there is a first $k\leqslant s-2\leqslant m-2$ with $R_{k+1}\leqslant\gamma R_k$.  By minimality,
\[
       R_k\geqslant \gamma^{m-2}(1-\eps)M.
\]

Let $v=x_k/\|x_k\|$ and $w=x_{k+1}/\|x_{k+1}\|$.  Thus $Tv=R_kw$ and, if $R_{k+1}>0$, $Tw=R_{k+1}u$ for some unit vector $u$; if $R_{k+1}=0$ the following error term is simply absent.  For $\zeta\in\T$ put $f(r)=\|v+r\zeta w\|$.  On $[2,7]$ we have $f(r)\geqslant r-1>0$, $f(2)\leqslant3$, and $f(7)\geqslant6$.  Since $f$ is convex,
\[
       \int_2^7\frac{f'_+(s)}{f(s)}\,ds
       =\log f(7)-\log f(2)\geqslant\log2,
\]
so there is a point $r\in[2,7]$ with $f'_+(r)/f(r)\geqslant\eta$.  Moreover $f(r)\geqslant r-1$, whence $r/f(r)\leqslant2$.  Using
\[
 (I+t\zeta T)(v+r\zeta w)=v+(r+tR_k)\zeta w+tr\zeta^2R_{k+1}u
\]
when $R_{k+1}>0$ and the analogous identity without the last term when $R_{k+1}=0$, we get
\[
       \|I+t\zeta T\|\geqslant
       \frac{f(r+tR_k)-trR_{k+1}}{f(r)}.
\]
Formula \eqref{eq:support} therefore yields
\[
       \sup_{z\in V(T)}\operatorname{Re}(\zeta z)
       \geqslant R_k\frac{f'_+(r)}{f(r)}-R_{k+1}\frac{r}{f(r)}
       \geqslant \eta R_k-2R_{k+1}
       \geqslant \frac\eta2 R_k.
\]
Letting $\eps\downarrow0$ gives the disc inclusion with
$\delta_m=(\eta/2)\gamma^{m-2}$.
\end{proof}

\subsection{Uniform finiteness in every fixed degree}

The next ratio-drop estimate is the large-norm ingredient in the proof that every \(\Gamma_n\) is finite.

\begin{lemma}\label{lem:general-large-disc}
Fix an integer \(m\geqslant2\), and put
\[
       \eta=\frac{\log2}{5},\qquad
       \gamma=\frac{\eta}{4},\qquad
       \delta_m=\frac{\eta}{2}\gamma^{m-2},\qquad
       B_m=2^m\gamma^{-(m-1)}.
\]
Let \(T\in\calB(X)\), and suppose that \(q(T)=0\), where
\[
       q(z)=z^m+c_{m-1}z^{m-1}+\cdots+c_0
\]
is monic and all roots of \(q\) belong to \(\overline\D\).  If \(\|T\|>B_m\), then
\[
       \delta_m\|T\|\overline\D\subseteq V(T).
\]
\end{lemma}

\begin{proof}
Put \(M=\|T\|\), fix \(\zeta\in\T\), and choose \(\eps>0\) so small that \((1-\eps)M>B_m\).  Choose \(x_0\in X\) with \(\|x_0\|=1\) and \(\|Tx_0\|>(1-\eps)M\), and set \(x_j=T^jx_0\).  Whenever \(x_j\neq0\), write
\[
       R_j=\frac{\|x_{j+1}\|}{\|x_j\|}.
\]
If some \(x_s\) with \(1\leqslant s\leqslant m\) is zero, take the least such \(s\).  Then \(R_{s-1}=0\), so there is a first \(k\leqslant s-2\leqslant m-2\) with \(R_{k+1}\leqslant\gamma R_k\).

Suppose instead that \(x_0,\ldots,x_m\) are all non-zero.  If no such drop occurred for \(0\leqslant k\leqslant m-2\), then
\[
       R_j>\gamma^jR_0\qquad(0\leqslant j\leqslant m-1).
\]
In particular, \(R_j>\gamma^{m-2}R_0>1\) for \(0\leqslant j\leqslant m-2\).  Since the roots of \(q\) lie in \(\overline\D\),
\[
       \sum_{j=0}^{m-1}|c_j|\leqslant2^m-1<2^m.
\]
The relation \(q(T)x_0=0\) therefore gives
\[
\begin{split}
       R_{m-1}
       &\leqslant |c_{m-1}|+
       \sum_{j=0}^{m-2}\frac{|c_j|}{R_jR_{j+1}\cdots R_{m-2}}\\
       &\leqslant\sum_{j=0}^{m-1}|c_j|<2^m,
\end{split}
\]
whereas \(R_{m-1}>\gamma^{m-1}R_0>2^m\), a contradiction.  Thus in every case there is a first \(k\leqslant m-2\) such that
\[
       R_{k+1}\leqslant\gamma R_k,
       \qquad
       R_k\geqslant\gamma^{m-2}(1-\eps)M.
\]
The latter inequality and the definition of \(B_m\) imply \(R_k>1\).  Hence the unit vectors
\[
       v=\frac{x_k}{\|x_k\|},\qquad
       w=\frac{x_{k+1}}{\|x_{k+1}\|}
\]
are linearly independent: otherwise \(v\) would be an eigenvector of \(T\) with eigenvalue of modulus \(R_k>1\), contrary to \(\sigma(T)\subseteq\overline\D\).  Set
\[
       f(r)=\|v+r\zeta w\|\qquad(r\geqslant0).
\]
This function is convex and satisfies $f(r)\geqslant r-1>0$ on $[2,7]$.  We have \(Tv=R_kw\), and either \(Tw=R_{k+1}u\) for a unit vector \(u\), or the corresponding term is zero.  As in Proposition~\ref{prop:nilpotents}, some \(r\in[2,7]\) satisfies
\[
       \frac{f'_+(r)}{f(r)}\geqslant\eta,
       \qquad
       \frac r{f(r)}\leqslant2.
\]
For \(t>0\),
\[
       (I+t\zeta T)(v+r\zeta w)
       =v+(r+tR_k)\zeta w+tr\zeta^2R_{k+1}u,
\]
with the last term omitted when \(R_{k+1}=0\).  Hence
\[
       \|I+t\zeta T\|
       \geqslant
       \frac{f(r+tR_k)-trR_{k+1}}{f(r)}.
\]
Subtracting \(1\), dividing by \(t\), and using
Formula~\eqref{eq:support}, we obtain
\[
\begin{split}
       \sup_{z\in V(T)}\operatorname{Re}(\zeta z)
       &\geqslant
       R_k\frac{f'_+(r)}{f(r)}-R_{k+1}\frac r{f(r)}\\
       &\geqslant \eta R_k-2R_{k+1}
       \geqslant\frac\eta2R_k
       \geqslant\delta_m(1-\eps)M.
\end{split}
\]
Letting \(\eps\downarrow0\) and using Lemma~\ref{lem:support-containment} proves the disc inclusion.
\end{proof}

\begin{lemma}\label{lem:general-large-calculus}
Fix \(m\geqslant2\), and retain the constants from Lemma~\ref{lem:general-large-disc}.  Let \(T\in\calB(X)\) have monic minimal polynomial of degree \(m\), all of whose roots belong to \(\overline\D\).  If
\[
       M:=\|T\|>R_m:=\max\{B_m,2/\delta_m,1\},
\]
then
\[
       \|p(T)\|\leqslant E_m\sup_{z\in V(T)}|p(z)|
       \qquad(p\in\C[z]),
\]
where one may take
\[
       E_m=1+2\sum_{j=1}^{m-1}\left(\frac4{\delta_m}\right)^j.
\]
\end{lemma}

\begin{proof}
Put \(K=\sup_{V(T)}|p|\), and assume \(K>0\).  By Lemma~\ref{lem:general-large-disc}, the disc \(\varrho\overline\D\), where \(\varrho=\delta_mM\geqslant2\), is contained in \(V(T)\).  List the roots of the minimal polynomial, with multiplicity, as
\(\lambda_1,\ldots,\lambda_m\), and put
\[
       Q_0(z)=1,
       \qquad
       Q_j(z)=\prod_{i=1}^j(z-\lambda_i)
       \quad(1\leqslant j\leqslant m-1).
\]
The Newton--Hermite polynomial from \eqref{eq:newton-hermite-form} is
\[
       h(z)=\sum_{j=0}^{m-1}d_jQ_j(z),
       \qquad
       d_0=p(\lambda_1),
       \quad
       d_j=p[\lambda_1,\ldots,\lambda_{j+1}]\quad(j\geqslant1).
\]
It has the same interpolation data as \(p\) at every root, with the
multiplicity prescribed by the minimal polynomial.  Hence \(p-h\) is
divisible by that minimal polynomial and
\[
       p(T)=h(T)=\sum_{j=0}^{m-1}d_jQ_j(T).
\]
For \(|z|=\varrho\), every root satisfies
\(|z-\lambda_i|\geqslant\varrho-1\), while \(|p(z)|\leqslant K\) because
\(\varrho\overline\D\subseteq V(T)\).  Formula~\eqref{eq:cauchy-divided-difference}
therefore yields, for \(1\leqslant j\leqslant m-1\),
\[
\begin{split}
       |d_j|
       &\leqslant
       \frac{1}{2\pi}(2\pi\varrho)
       \frac{K}{(\varrho-1)^{j+1}}\\
       &=\frac{K\varrho}{(\varrho-1)^{j+1}}.
\end{split}
\]
Also \(|d_0|\leqslant K\), because \(\lambda_1\in\sigma(T)\subseteq V(T)\), and
\[
       \|Q_j(T)\|\leqslant(M+1)^j.
\]
Since \(\varrho\geqslant2\) and \(M\geqslant1\),
\[
       \frac{\varrho}{\varrho-1}\leqslant2,
       \qquad
       \frac{M+1}{\varrho-1}\leqslant\frac4{\delta_m}.
\]
Consequently,
\[
\begin{split}
       \|p(T)\|
       &\leqslant |d_0|+\sum_{j=1}^{m-1}|d_j|\,\|Q_j(T)\|\\
       &\leqslant K+2K\sum_{j=1}^{m-1}
       \left(\frac{M+1}{\varrho-1}\right)^j\\
       &\leqslant
       \left[1+2\sum_{j=1}^{m-1}
       \left(\frac4{\delta_m}\right)^j\right]K
       =E_mK.
\end{split}
\]
\end{proof}

\begin{lemma}\label{lem:root-partition}
Let \(\Lambda\subseteq\C\) consist of \(s\geqslant2\) distinct points and have diameter one.  Then there is a partition \(\Lambda=\Lambda_0\mathbin{\dot\cup}\Lambda_1\) into non-empty sets such that
\[
       \operatorname{dist}(\Lambda_0,\Lambda_1)\geqslant\frac1{s-1}.
\]
\end{lemma}

\begin{proof}
Consider the complete graph on \(\Lambda\), with edge weight
\(|\lambda-\mu|\) on the edge joining \(\lambda\) and \(\mu\), and choose a
minimum spanning tree.  We use the standard cut optimality condition: each
edge of a minimum spanning tree has minimum weight among the edges crossing
the cut obtained by deleting it; see \cite[Theorem~6.3(c)]{KorteVygen}.
Choose two points of \(\Lambda\) at distance one.  Their path in the tree has
at most \(s-1\) edges, so it contains an edge \(e\) of length at least
\(1/(s-1)\).  Deleting \(e\) gives two non-empty vertex sets
\(\Lambda_0\) and \(\Lambda_1\).  By the cut optimality condition, every
edge joining \(\Lambda_0\) to \(\Lambda_1\) has length at least that of
\(e\).  Hence
\(
\operatorname{dist}(\Lambda_0,\Lambda_1)\geqslant1/(s-1)
\), as required.
\end{proof}

\begin{lemma}\label{lem:uniform-spectral-splitting}
For \(m\geqslant2\) and \(d,R>0\), put
\[
\begin{split}
       \beta_m(d)
       &=(2^m\sqrt m)^{m-1}
         \min\{1,d\}^{-m^2/4},\\
       S_m(R)
       &=\sum_{k=0}^{m-1}R^k,\\
       H_m(d,R)
       &=2\beta_m(d)S_m(R)(1+R)^m.
\end{split}
\]
Suppose that \(T\in\calB(X)\), \(\|T\|\leqslant R\), and
\[
       q(T)=0,
       \qquad
       q=q_0q_1,
\]
where \(q_0,q_1\) are non-constant monic polynomials,
\(\deg q\leqslant m\), all their roots belong to \(\overline\D\), and the
two root sets have distance at least \(d\).  Then there is a projection
\(P\in\calB(X)\), commuting with \(T\), such that
\[
       q_0(T)P=0,
       \qquad
       q_1(T)(I-P)=0,
       \qquad
       \|P\|+\|I-P\|\leqslant H_m(d,R).
\]
\end{lemma}

\begin{proof}
Write \(r=\deg q_0\), \(s=\deg q_1\), and \(N=r+s\leqslant m\).
Since the two polynomials are coprime, there are unique polynomials \(u,v\),
with \(\deg u<s\) and \(\deg v<r\), such that
\[
       uq_0+vq_1=1.
\]
Their coefficient vectors solve the Sylvester linear system associated with
\(q_0\) and \(q_1\).  After ordering coefficients by increasing powers, the
right-hand side of this system is a coordinate vector; hence each numerator
in Cramer's rule is a single cofactor of the Sylvester matrix.  Write
\[
       q_j(z)=z^{m_j}+a_{j,m_j-1}z^{m_j-1}+\cdots+a_{j,0}
       \qquad(j=0,1).
\]
Because all roots lie in \(\overline\D\), the elementary-symmetric-function
formula gives
\[
       |a_{j,k}|\leqslant\binom{m_j}{k}\leqslant2^m.
\]
Thus every entry of the \(N\times N\) Sylvester matrix
\(S(q_0,q_1)\) has modulus at most \(2^m\).  By Hadamard's inequality,
every cofactor has modulus at most
\[
       (2^m\sqrt{N-1})^{N-1}
       \leqslant(2^m\sqrt m)^{m-1}.
\]
Let \(\alpha_1,\ldots,\alpha_r\) and
\(\beta_1,\ldots,\beta_s\) be the roots of \(q_0\) and \(q_1\),
respectively, listed with multiplicity.  Then
\[
\begin{split}
       |\det S(q_0,q_1)|
       &=|\operatorname{Res}(q_0,q_1)|\\
       &=\prod_{i=1}^r\prod_{j=1}^s
         |\alpha_i-\beta_j|
       \geqslant d^{rs}.
\end{split}
\]  Since
\(rs\leqslant m^2/4\), this determinant is at least
\(\min\{1,d\}^{m^2/4}\).  Cramer's rule therefore shows that every
coefficient of \(u\) and \(v\) has modulus at most \(\beta_m(d)\).

Put
\[
       P=v(T)q_1(T),
       \qquad
       I-P=u(T)q_0(T).
\]
If \(e=vq_1\), then the B\'ezout identity gives
\[
       e\equiv1\pmod{q_0},
       \qquad
       e\equiv0\pmod{q_1}.
\]
Hence \(e^2-e\) is divisible by both coprime factors and therefore by
\(q_0q_1=q\); thus \(P^2=P\).  Moreover,
\[
       q_0(T)P=v(T)q(T)=0,
       \qquad
       q_1(T)(I-P)=u(T)q(T)=0.
\]
The coefficient estimate gives
\[
       \|u(T)\|,\ \|v(T)\|
       \leqslant\beta_m(d)S_m(R),
\]
while the root condition implies
\[
       \|q_0(T)\|,\ \|q_1(T)\|
       \leqslant(1+R)^m.
\]
Consequently,
\[
       \|P\|+\|I-P\|
       \leqslant2\beta_m(d)S_m(R)(1+R)^m
       =H_m(d,R).
\]
\end{proof}

\begin{theorem}\label{thm:all-degrees}
Set
\[
       G_1=1,
       \qquad
       G_2=\sqrt{1+(2\mathrm e-1)^2},
\]
and, for \(n\geqslant3\), define recursively
\begin{equation}\label{eq:recursive-upper-bound}
       G_n=\max\left\{
       G_{n-1},\ N_n,\ E_n,\
       H_n\!\left(\frac1{n-1},R_n\right)G_{n-1}
       \right\},
\end{equation}
where \(N_n\), \(E_n\), \(R_n\) and \(H_n\) are the explicit constants
from Proposition~\ref{prop:nilpotents}, Lemma~\ref{lem:general-large-calculus}
and Lemma~\ref{lem:uniform-spectral-splitting}.  Then
\[
       \Gamma_1=1,
       \qquad
       2n-1\leqslant\Gamma_n\leqslant G_n<\infty
       \quad(n\geqslant2).
\]
\end{theorem}

\begin{proof}
Every element of algebraic degree one is scalar, so \(\Gamma_1=1\).
Theorem~\ref{thm:degree-two} gives \(\Gamma_2\leqslant G_2\).  Assume
inductively that \(\Gamma_{n-1}\leqslant G_{n-1}\), where \(n\geqslant3\),
and let \(a\) have algebraic degree at most \(n\).  Elements of smaller
degree are already bounded by \(G_{n-1}\), so suppose that the minimal
polynomial has degree exactly \(n\).  By Lemma~\ref{lem:left-regular}, we may
pass to the left multiplication operator and work with
\(T\in\calB(X)\) without changing its norm, minimal polynomial, numerical
range or spectral constant.

If the minimal polynomial has only one distinct root, an affine translation
reduces the relation to \(T^n=0\), and Proposition~\ref{prop:nilpotents}
gives the bound \(N_n\).  Suppose that there are at least two distinct roots.
Choose two roots at maximal distance and apply an affine transformation which
sends them to \(0\) and \(1\).  Every root of the transformed minimal
polynomial then belongs to \(\overline\D\), and \(\Psi\) is unchanged.  Put
\(M=\|T\|\).

If \(M>R_n\), Lemma~\ref{lem:general-large-calculus} gives the bound \(E_n\).
Assume therefore that \(M\leqslant R_n\).  Let \(\Lambda\) be the set of
distinct roots, say \(|\Lambda|=s\).  By Lemma~\ref{lem:root-partition}, write
\(\Lambda=\Lambda_0\mathbin{\dot\cup}\Lambda_1\) with
\[
       \operatorname{dist}(\Lambda_0,\Lambda_1)
       \geqslant\frac1{s-1}\geqslant\frac1{n-1}.
\]
Factor the minimal polynomial as \(q=q_0q_1\), assigning to \(q_j\) all
multiplicities of the roots in \(\Lambda_j\).  Lemma~\ref{lem:uniform-spectral-splitting}
gives a projection \(P\) commuting with \(T\), such that
\[
       \|P\|+\|I-P\|
       \leqslant H_n\!\left(\frac1{n-1},R_n\right),
\]
and such that the restrictions of \(T\) to
\[
       Y_0=PX,
       \qquad
       Y_1=(I-P)X
\]
have algebraic degrees at most \(\deg q_0\) and \(\deg q_1\), respectively.
Both degrees are at most \(n-1\).

Let \(K=\sup_{z\in V(T)}|p(z)|>0\).  For each non-zero \(Y_j\),
Lemma~\ref{lem:invariant-restriction} shows that the numerical range of the
restriction is contained in \(V(T)\), and the induction hypothesis gives
\[
       \|p(T)|_{Y_j}\|\leqslant G_{n-1}K.
\]
If one of the spaces \(Y_j\) is zero, the corresponding term below is simply
omitted.  Consequently,
\[
\begin{split}
       \|p(T)\|
       &\leqslant
       \|p(T)|_{Y_0}\|\,\|P\|
       +\|p(T)|_{Y_1}\|\,\|I-P\|\\
       &\leqslant
       H_n\!\left(\frac1{n-1},R_n\right)G_{n-1}K.
\end{split}
\]
The four alternatives are exactly the four terms in
\eqref{eq:recursive-upper-bound}; hence \(\Gamma_n\leqslant G_n\).
The lower bound is Proposition~\ref{prop:gamma-n-lower}.
\end{proof}

\begin{corollary}\label{cor:matrix-algebras}
Let \(d\geqslant1\), and equip \(M_d(\C)\) with any unital submultiplicative norm for which \(\|I\|=1\).  Then
\[
       \Psi_{M_d(\C)}\leqslant\Gamma_d<\infty.
\]
The same conclusion holds for every unital subalgebra of \(M_d(\C)\),
endowed with the inherited norm.
\end{corollary}

\begin{proof}
Every element of \(M_d(\C)\), and hence of any of its subalgebras, has minimal-polynomial degree at most \(d\).
\end{proof}

Theorem~\ref{thm:all-degrees} and Corollary~\ref{cor:matrix-algebras} therefore answer both parts of Question~1 in \cite[Section~8]{BHDVW}: the fixed-degree constants and the constants of finite matrix algebras are finite.

\begin{remark}\label{rem:quantitative-gap}
The proof leaves a substantial quantitative gap.  Proposition~\ref{prop:gamma-n-lower}
gives the linear lower estimate
\[
       2n-1\leqslant\Gamma_n,
\]
whereas the constants produced by the spectral-splitting induction grow much
faster.  It is therefore natural to ask whether
\[
       \Gamma_2=3
       \qquad\text{and, more generally,}\qquad
       \Gamma_n=2n-1\quad(n\geqslant2).
\]
The present arguments do not decide even the first equality.  More modestly,
it would already be interesting to determine the correct order of growth of
\(\Gamma_n\).
\end{remark}

\section{A gap phenomenon for \texorpdfstring{\(C^*\)}{C-star}-algebras}\label{sec:cstar-gap}

For a bounded operator \(T\) on a complex Hilbert space \(H\), write
\[
       W(T)=\{\langle T\xi,\xi\rangle_H:\xi\in H,\ \|\xi\|=1\}
\]
for its spatial numerical range.  We denote by \(C_{\mathrm{Cr}}\) the optimal
universal constant in
\[
       \|p(T)\|\leqslant C_{\mathrm{Cr}}
       \sup_{z\in W(T)}|p(z)|
       \qquad(T\in\calB(H),\ p\in\C[z]),
\]
where, as usual, polynomials with zero denominator are omitted.  Crouzeix and
Palencia proved
\begin{equation}\label{eq:crouzeix-palencia}
       C_{\mathrm{Cr}}\leqslant1+\sqrt2;
\end{equation}
see \cite[Theorem~3.1 and formula~(3)]{CrouzeixPalencia}.

\begin{theorem}\label{thm:cstar-gap}
Let \(\calA\) be a unital \(C^*\)-algebra.  Then
\[
       \Psi_{\calA}=1
       \quad\Longleftrightarrow\quad
       \calA\text{ is commutative}.
\]
If \(\calA\) is non-commutative, then
\[
       2\leqslant\Psi_{\calA}\leqslant C_{\mathrm{Cr}}
       \leqslant1+\sqrt2.
\]
Consequently, no value strictly between \(1\) and \(2\) occurs as the
constant of a unital \(C^*\)-algebra.
\end{theorem}

\begin{proof}
Suppose first that \(\calA\) is commutative.  Every \(a\in\calA\) is normal,
so the continuous functional calculus gives
\[
       \|p(a)\|=\max_{\lambda\in\sigma(a)}|p(\lambda)|
       \leqslant\sup_{z\in V(a,\calA)}|p(z)|.
\]
The constant polynomial \(1\) gives the reverse inequality, and hence
\(\Psi_{\calA}=1\).

Assume now that \(\calA\) is non-commutative.  The discussion preceding
\cite[Theorem~2.1]{AlaminosSquareZero} records that \(C^*\)-algebras are
essential and have property~B, while
\cite[Proposition~6.1]{AlaminosSquareZero} shows that every non-commutative
essential Banach algebra with property~B contains a non-zero square-zero
element.  Hence every non-commutative \(C^*\)-algebra contains such an
element.
Choose such an \(x\in\calA\), so that \(x^2=0\).  Choose a
faithful unital \(*\)-representation of \(\calA\) and identify \(\calA\) with
its image in \(\calB(H)\).  Subalgebra invariance gives
\[
       V(x,\calA)=V(x,\calB(H))=\overline{W(x)}.
\]
Let
\[
       w(x)=\sup_{z\in W(x)}|z|
\]
be the Hilbert-space numerical radius.  From the definition,
\[
\begin{split}
       w(x)
       &=\sup_{\|\xi\|=1}\sup_{\theta\in\mathbb R}
          \operatorname{Re}\langle
             \mathrm e^{\mathrm i\theta}x\xi,\xi\rangle_H\\
       &=\sup_{\theta\in\mathbb R}
          \bigl\|\operatorname{Re}(\mathrm e^{\mathrm i\theta}x)\bigr\|.
\end{split}
\]
For the last equality, one uses the Rayleigh-quotient formula for a
self-adjoint operator and replaces \(\theta\) by \(\theta+\pi\) when the
negative spectral endpoint has larger modulus.  Since \(x^2=(x^*)^2=0\),
\[
\begin{split}
       4\operatorname{Re}(\mathrm e^{\mathrm i\theta}x)^2
       &=(\mathrm e^{\mathrm i\theta}x+
          \mathrm e^{-\mathrm i\theta}x^*)^2\\
       &=xx^*+x^*x.
\end{split}
\]
The positive elements \(xx^*\) and \(x^*x\) are orthogonal, since
\[
       (xx^*)(x^*x)=x(x^*)^2x=0,
       \qquad
       (x^*x)(xx^*)=x^*x^2x^*=0.
\]
Therefore
\[
       \|xx^*+x^*x\|
       =\max\{\|xx^*\|,\|x^*x\|\}
       =\|x\|^2,
\]
and hence
\[
       w(x)=\frac{\|x\|}{2}.
\]
Testing the definition of \(\Psi(x,\calA)\) on \(p(z)=z\) gives
\[
       \Psi_{\calA}\geqslant\Psi(x,\calA)
       \geqslant\frac{\|x\|}{w(x)}=2.
\]

For the upper bound, let \(a\in\calA\), choose a faithful unital
\(*\)-representation \(\pi:\calA\to\calB(H)\), and let \(p\in\C[z]\).  By
subalgebra invariance,
\[
       V(a,\calA)=V(\pi(a),\calB(H))=\overline{W(\pi(a))}.
\]
The definition of \(C_{\mathrm{Cr}}\) therefore yields
\[
\begin{split}
       \|p(a)\|
       &=\|p(\pi(a))\|\\
       &\leqslant C_{\mathrm{Cr}}
          \sup_{z\in W(\pi(a))}|p(z)|\\
       &=C_{\mathrm{Cr}}
          \sup_{z\in V(a,\calA)}|p(z)|.
\end{split}
\]
Thus \(\Psi_{\calA}\leqslant C_{\mathrm{Cr}}\), and
\eqref{eq:crouzeix-palencia} completes the proof.
\end{proof}

The Jiang--Su algebra is included here for a structural reason.  It was
introduced in \cite{JiangSu}, and its relatively prime dimension-drop
subalgebras embed unitally into it; see
\cite[Theorem~2.2 and Proposition~3.3]{RordamWinter}.  Since the constant of
a dimension-drop algebra turns out to be the corresponding matrix Crouzeix
constant, these embeddings allow one canonical simple \(C^*\)-algebra to
detect matrix sizes tending to infinity.  The next proposition therefore
identifies \(\Psi_{\mathcal Z}\) exactly with the universal Crouzeix constant.

To state the result conveniently, let
\[
       C_d=
       \sup_{A\in M_d(\C)}
       \sup_{\substack{p\in\C[z]\\
               \sup_{z\in W(A)}|p(z)|>0}}
       \frac{\|p(A)\|}{\sup_{z\in W(A)}|p(z)|}
       \qquad(d\geqslant1).
\]

\begin{proposition}\label{prop:jiang-su}
Let \(\mathcal Z\) denote the Jiang--Su algebra introduced in
\cite{JiangSu}.  Then
\[
       \Psi_{\mathcal Z}=C_{\mathrm{Cr}}.
\]
More precisely, if \(r,s\geqslant2\) and
\[
 Z_{r,s}=\bigl\{f\in C([0,1],M_r\otimes M_s):
 f(0)\in M_r\otimes1_s,\quad f(1)\in1_r\otimes M_s\bigr\}
\]
is the dimension-drop algebra, then
\[
       \Psi_{Z_{r,s}}=C_{rs}.
\]
\end{proposition}

\begin{proof}
We first record that
\begin{equation}\label{eq:crouzeix-finite-matrices}
       C_{\mathrm{Cr}}=\sup_{d\geqslant1}C_d.
\end{equation}
The sequence \((C_d)\) is non-decreasing.  Indeed, if \(A\in M_d(\C)\),
\(e>d\), and \(\lambda\in W(A)\), then
\[
       B=A\oplus\lambda I_{e-d}\in M_e(\C)
\]
satisfies \(W(B)=W(A)\) and \(\|p(B)\|\geqslant\|p(A)\|\) for every
polynomial \(p\).  Hence \(C_d\leqslant C_e\).

It remains to justify that finite matrices detect the universal constant.  Let
\(T\in\calB(H)\), let \(p\) have degree \(m\), and fix \(\varepsilon>0\).
Choose a unit vector \(\xi\in H\) such that
\[
       \|p(T)\xi\|>(1-\varepsilon)\|p(T)\|,
\]
and put
\[
       E=\operatorname{span}\{\xi,T\xi,\ldots,T^m\xi\},
       \qquad A=P_ET|_E,
\]
where \(P_E\) is the orthogonal projection onto \(E\).  For
\(0\leqslant k\leqslant m\), induction gives \(A^k\xi=T^k\xi\), and hence
\(p(A)\xi=p(T)\xi\).  Moreover, \(W(A)\subseteq W(T)\).  Therefore
\[
\begin{split}
       (1-\varepsilon)\|p(T)\|
       &<\|p(A)\|\\
       &\leqslant C_{m+1}\sup_{z\in W(A)}|p(z)|\\
       &\leqslant C_{m+1}\sup_{z\in W(T)}|p(z)|.
\end{split}
\]
Letting \(\varepsilon\downarrow0\) proves \eqref{eq:crouzeix-finite-matrices}.

We next identify the constant of \(Z_{r,s}\).  If \(f\in Z_{r,s}\), then for
every polynomial \(p\),
\[
\begin{split}
       \|p(f)\|
       &=\max_{0\leqslant t\leqslant1}\|p(f(t))\|\\
       &\leqslant C_{rs}
          \max_{0\leqslant t\leqslant1}
          \sup_{z\in W(f(t))}|p(z)|\\
       &\leqslant C_{rs}\sup_{z\in V(f,Z_{r,s})}|p(z)|.
\end{split}
\]
The last inequality holds because every fibre vector state is a state of
\(Z_{r,s}\).  Thus \(\Psi_{Z_{r,s}}\leqslant C_{rs}\).

For the reverse inequality, take \(A\in M_{rs}(\C)\), choose
\(\lambda\in W(A)\), and let \(h:[0,1]\to[0,1]\) be continuous with
\(h(0)=h(1)=0\) and \(h(t_0)=1\) at some \(t_0\in(0,1)\).  Define
\[
       f(t)=\lambda I_{rs}+h(t)(A-\lambda I_{rs}).
\]
Then \(f\in Z_{r,s}\), and convexity of \(W(A)\) gives
\[
       W(f(t))=\lambda+h(t)(W(A)-\lambda)\subseteq W(A)
       \qquad(0\leqslant t\leqslant1).
\]
On the other hand, \(f(t_0)=A\).  Since \(Z_{r,s}\) is a unital
\(C^*\)-subalgebra of \(C([0,1],M_{rs})\), subalgebra invariance and the
description of the pure states of \(C([0,1],M_{rs})\) as fibre vector states
give
\[
\begin{split}
       V(f,Z_{r,s})
       &=V(f,C([0,1],M_{rs}))\\
       &=\conv\bigcup_{0\leqslant t\leqslant1}W(f(t))
        =W(A).
\end{split}
\]
Also \(\|p(f)\|\geqslant\|p(f(t_0))\|=\|p(A)\|\).  Taking suprema over
\(A\) and \(p\) yields \(\Psi_{Z_{r,s}}\geqslant C_{rs}\), and hence
\(\Psi_{Z_{r,s}}=C_{rs}\).

For every \(n\geqslant2\), the relatively prime dimension-drop algebra
\(Z_{n,n+1}\) embeds unitally into \(\mathcal Z\); see
\cite[Proposition~3.3]{RordamWinter}.  Subalgebra invariance therefore gives
\[
       \Psi_{\mathcal Z}\geqslant\Psi_{Z_{n,n+1}}
       =C_{n(n+1)}.
\]
Since \((C_d)\) is non-decreasing and \(n(n+1)\to\infty\),
\eqref{eq:crouzeix-finite-matrices} implies
\(\Psi_{\mathcal Z}\geqslant C_{\mathrm{Cr}}\).  The reverse inequality is
Theorem~\ref{thm:cstar-gap}.
\end{proof}

\begin{remark}\label{rem:cstar-square-zero}
The square-zero element in the proof is explicit under several familiar
hypotheses.  If \(p\in\calA\) is a non-central projection, then at least one of
the corners \(p\calA(1-p)\) and \((1-p)\calA p\) is non-zero; every element
of either corner has square zero.  If \(v\in\calA\) is a proper isometry, so
\(v^*v=1\) and \(vv^*\neq1\), put \(q=1-vv^*\).  Then \(qv=0\) and
\[
       x=vq,
       \qquad
       x^*x=q\neq0,
       \qquad
       x^2=v(qv)q=0.
\]
Theorem~\ref{thm:cstar-gap} shows, however, that neither a non-central
projection nor a proper isometry is needed: non-commutativity itself is the
intrinsic condition.
\end{remark}

The operator form of Crouzeix's conjecture is the assertion
\(C_{\mathrm{Cr}}=2\).  The preceding theorem gives an equivalent
\(C^*\)-algebraic formulation.

\begin{proposition}\label{prop:cstar-crouzeix}
The following assertions are equivalent:
\begin{enumerate}
\item[(i)] \(C_{\mathrm{Cr}}=2\);
\item[(ii)] \(\Psi_{\calA}=2\) for every non-commutative unital
      \(C^*\)-algebra \(\calA\);
\item[(iii)] \(\Psi_{\mathcal Z}=2\) for the Jiang--Su algebra \(\mathcal Z\).
\end{enumerate}
\end{proposition}

\begin{proof}
The implication \textup{(i)}\(\Rightarrow\)\textup{(ii)} follows immediately
from Theorem~\ref{thm:cstar-gap}, and \textup{(ii)}\(\Rightarrow\)\textup{(iii)}
is clear because \(\mathcal Z\) is non-commutative.  Finally,
Proposition~\ref{prop:jiang-su} gives
\(\Psi_{\mathcal Z}=C_{\mathrm{Cr}}\), so
\textup{(iii)}\(\Rightarrow\)\textup{(i)}.
\end{proof}

\section{Counterexamples in Banach spaces}\label{sec:counterexamples}

\subsection{A contractive disc calculus with infinite constant}\label{sec:polynomially-bounded}

Let $L(x_1,x_2,\ldots)=(x_2,x_3,\ldots)$ be the unilateral left shift.  We first recall Bohr's inequality and a short interpolation lemma; the polynomially bounded example then follows at once.

\begin{lemma}[Bohr's inequality \cite{Bohr}]\label{lem:bohr}
Let $f(z)=\sum_{k=0}^m a_kz^k$ be a polynomial with $\sup_{|z|\leqslant1}|f(z)|\leqslant1$.  Then
\[
       \sum_{k=0}^m |a_k|3^{-k}\leqslant1.
\]
\end{lemma}

Example~6.2 of \cite{BHDVW} uses the same direct-sum geometry and obtains the polynomial bound $2\sqrt3/3$.  The following interpolation estimate is the additional ingredient which lowers that bound to one.

\begin{lemma}\label{lem:interpolation}
Let $1\leqslant p\leqslant\infty$ and set $r_p=3^{-|1-2/p|}$, with the usual interpretation when $p=\infty$.  Then, for every polynomial $f$,
\[
       \|f(r_pL)\|_{\ell_p\to\ell_p}\leqslant\sup_{|z|\leqslant1}|f(z)|.
\]
Consequently, if $r_p>1/2$, then $\|f((1/2)L)\|_{\ell_p\to\ell_p}\leqslant\sup_{|z|\leqslant1}|f(z)|$.
\end{lemma}

\begin{proof}
By homogeneity assume that $\sup_{\overline\D}|f|\leqslant1$ and write $f(z)=\sum_{k=0}^ma_kz^k$.  On $\ell_2$, von Neumann's inequality \cite{vN} gives $\|f(\omega L)\|\leqslant1$ for $|\omega|=1$.  On both $\ell_1$ and $\ell_\infty$, Lemma~\ref{lem:bohr} and $\|L^k\|\leqslant1$ give $\|f(3^{-1}\omega L)\|\leqslant1$.

We regard $\ell_1$, $\ell_2$ and $\ell_\infty$ as compatible Banach spaces inside the vector space of all scalar sequences.  For the complex method, the counting-measure case of \cite[Theorem~5.1.1, pp.~106--107]{BerghLofstrom} gives, with equality of norms,
\[
       [\ell_1,\ell_2]_\theta=\ell_p,
       \qquad \frac1p=1-\frac\theta2,
\]
and
\[
       [\ell_2,\ell_\infty]_\theta=\ell_p,
       \qquad \frac1p=\frac{1-\theta}{2}.
\]

We use Stein's interpolation theorem for analytic families of operators in the precise form recorded in \cite[Exercise~1.6.11, p.~17]{BerghLofstrom}.  In the special case in which both boundary norms are at most one, its proof---an adaptation of the Riesz--Thorin argument in \cite[Theorem~1.1.1, pp.~2--4]{BerghLofstrom} using the three-lines theorem \cite[Lemma~1.1.2, pp.~4--5]{BerghLofstrom}---gives interpolated norm at most one.  In the present situation the operator families are finite sums of entire scalar functions times powers of $L$, so $z\mapsto S_zx$ is analytic for every $x$, and the uniform boundary estimates below make the growth hypothesis in Exercise~1.6.11 automatic.

For $1\leqslant p\leqslant2$, apply this theorem to
\[
       S_z=f(3^{-(1-z)}L)\qquad(0\leqslant\operatorname{Re}z\leqslant1).
\]
On the line $z=it$ one has $3^{-(1-it)}=3^{-1}\omega_t$ with $|\omega_t|=1$, so the $\ell_1$-operator norm is at most one.  On the line $z=1+it$ one has $3^{it}=\omega_t$, so the $\ell_2$-operator norm is at most one.  At the interpolation point $\theta$ determined by
\[
       \frac1p=1-\frac\theta2,
\]
the theorem therefore gives
\[
       \|f(3^{-(1-\theta)}L)\|_{\ell_p\to\ell_p}\leqslant1,
       \qquad
       3^{-(1-\theta)}=3^{-(2/p-1)}.
\]

For $2\leqslant p\leqslant\infty$, apply the same result to
\[
       R_z=f(3^{-z}L).
\]
The line $z=it$ is controlled on $\ell_2$, while on $z=1+it$ one has $3^{-(1+it)}=3^{-1}\omega_t$ and hence the $\ell_\infty$-operator norm is at most one.  If $1/p=(1-\theta)/2$, then
\[
       \|f(3^{-\theta}L)\|_{\ell_p\to\ell_p}\leqslant1,
       \qquad
       3^{-\theta}=3^{-(1-2/p)}.
\]
This proves the first assertion, including the endpoint cases.  Finally, if $r_p>1/2$, apply it to $g(z)=f((2r_p)^{-1}z)$; then $g(r_pL)=f((1/2)L)$ and $\|g\|_{\infty,\overline\D}\leqslant\|f\|_{\infty,\overline\D}$.
\end{proof}

\begin{theorem}\label{thm:polynomially-bounded}
Let $1<p<\infty$, $p\neq2$, and suppose that $3^{-|1-2/p|}>1/2$.  Then there are a Banach space $X_p$ and an operator $T_p\in\calB(X_p)$ such that $\|T_p\|=1$, such that
\[
       \|q(T_p)\|\leqslant\sup_{|z|\leqslant1}|q(z)|\qquad(q\in\C[z]),
\]
but such that $\Psi(T_p,\calB(X_p))=\infty$.
\end{theorem}

\begin{proof}[Proof of Theorem~\ref{thm:polynomially-bounded}]
Fix $p$ as in the statement and let
\[
       X_p=\ell_2^2\oplus_\infty\ell_p,\qquad
       T_p=E\oplus \frac12L,\qquad
       E=\begin{pmatrix}0&1\\0&0\end{pmatrix}\in\calB(\ell_2^2).
\]
Then $\|T_p\|=1$.  Since $E$ is a Hilbert-space contraction and Lemma \ref{lem:interpolation} applies to $(1/2)L$ on $\ell_p$, we have $\|q(T_p)\|\leqslant\sup_{|z|\leqslant1}|q(z)|$ for every polynomial $q$.

The numerical range of $E$ on $\ell_2^2$ is $\frac12\overline\D$, and $V((1/2)L)\subseteq\frac12\overline\D$.  Formula~\eqref{eq:diagonal-sum}, together with subalgebra invariance of the algebraic numerical range inside $\calB(X_p)$, therefore gives $V(T_p)=\frac12\overline\D$.

Let $N_n$ be the left shift on $\ell_p^n$ and take the fixed Rudin--Shapiro Littlewood polynomial $q_n$ from the preliminaries.  If $p'$ denotes the conjugate exponent of $p$, then
\[
       \|q_n(N_n)\|_{\ell_p^n\to\ell_p^n}
       \geqslant n^{\max\{1/p,1/p'\}}.
\]
Indeed, evaluation at $e_n$ gives $n^{1/p}$, while the transpose evaluated at $e_1^*$ gives $n^{1/p'}$.  Let $Q_n:\ell_p^n\to\ell_p$ be the coordinate embedding onto the first $n$ coordinates and let $P_n:\ell_p\to\ell_p^n$ be the corresponding coordinate projection.  Then $P_nq_n(L)Q_n=q_n(N_n)$, and so $\|q_n(L)\|_{\ell_p\to\ell_p}\geqslant n^{\max\{1/p,1/p'\}}$.  Put $p_n(z)=q_n(2z)$.  Then $p_n((1/2)L)=q_n(L)$ and $\sup_{V(T_p)}|p_n|\leqslant\Delta\sqrt n$.  Since $p\neq2$,
\[
       \frac{\|p_n(T_p)\|}{\sup_{V(T_p)}|p_n|}
       \geqslant \Delta^{-1}n^{\max\{1/p,1/p'\}-1/2}\longrightarrow\infty.
\]
Thus $\Psi(T_p,\calB(X_p))=\infty$.
\end{proof}

\begin{corollary}\label{cor:contractive-polynomial-calculus}
There is an operator which is polynomially bounded with constant one and nevertheless has infinite algebraic numerical-range spectral constant.  Hence Question~2 of \cite[Section~8]{BHDVW} has an affirmative answer.
\end{corollary}

\begin{proof}[Proof of Corollary~\ref{cor:contractive-polynomial-calculus}]
Take $p=3$ in Theorem~\ref{thm:polynomially-bounded}, since $3^{-1/3}>1/2$.
\end{proof}

\begin{remark}
The interpolation lemma naturally suggests the shift Bohr radius
\[
       R_X=\sup\{r\in[0,1]:\|f(rS)\|\leqslant\sup_{|z|\leqslant1}|f(z)|\text{ for all polynomials }f\}
\]
for a Banach sequence space $X$ whose coordinate shift $S$ is bounded.  For an abstract Banach space this definition presupposes a specified coordinate structure, for example a Schauder basis $(e_j)$ and the associated shift $Se_j=e_{j+1}$ (or the corresponding left shift); unconditionality is not required merely to define $R_X$.  On a Banach sequence space the coordinate shift is already specified, but its boundedness must still be assumed.  The lemma says that $R_{\ell_p}\geqslant3^{-|1-2/p|}$.  Determining $R_{\ell_p}$ exactly would refine the distinction between polynomial boundedness and finite algebraic numerical-range spectral constant.
\end{remark}

\subsection{The canonical shift on combinatorial spaces}\label{sec:shift-obstructions}

\begin{theorem}\label{thm:combinatorial}
Let $\calF$ be a family of finite subsets of $\N$, containing sets of arbitrarily large cardinality and covering $\N$, and suppose it is spreading in the sense that $\{l_1<\cdots<l_m\}\in\calF$ and $l_j\leqslant k_j$ for all $j$ imply $\{k_1<\cdots<k_m\}\in\calF$.  Let $\calS$ be the completion of $c_{00}$ for
\[
       \|x\|_{\calS}=\sup_{F\in\calF}\sum_{i\in F}|x_i|.
\]
Then the left shift $L(x_1,x_2,\ldots)=(x_2,x_3,\ldots)$ is a contraction on $\calS$ and $\Psi(L,\calB(\calS))=\infty$.
\end{theorem}

Theorem~\ref{thm:combinatorial} strengthens \cite[Theorem~6.3]{BHDVW}.  That result proves $\Psi_{\calB(\calS)}=\infty$ by using a family of cut shifts, whereas the theorem above exhibits one fixed, canonical operator---the left shift---with infinite constant.  It therefore gives a negative answer to Question~3 of \cite[Section~8]{BHDVW}; the following corollary records the classical Schreier space explicitly.

\begin{corollary}\label{cor:schreier-shift}
Let
\[
       \calF_1=\{\varnothing\}\cup\{F\subseteq\N: F\neq\varnothing,\ |F|\leqslant \min F\}
\]
be the classical Schreier family, and let $\calS_1$ be the corresponding Schreier space.  Then the left shift is a contraction on $\calS_1$ and
$\Psi(L,\calB(\calS_1))=\infty$.
\end{corollary}

\begin{proof}[Proof of Theorem~\ref{thm:combinatorial}]
If $F\in\calF$, then $F+1\in\calF$ by spreading; hence the left shift is contractive on $c_{00}$ for the norm of $\calS$ and therefore extends to a contraction on $\calS$.

Call $n$ an admissible cardinality if some member of $\calF$ has cardinality $n$.  Let $n$ be admissible and choose $\{k_1<\cdots<k_n\}\in\calF$.  Since $k_j\leqslant k_n+j-1$, spreading gives the consecutive block $B_n=\{k_n,k_n+1,\ldots,k_n+n-1\}\in\calF$.  Let $Q_n:\ell_1^n\to\calS$ embed $\ell_1^n$ onto the coordinates in $B_n$, and let $P_n:\calS\to\ell_1^n$ be the coordinate projection onto $B_n$.  Then $\|Q_n\|=1$, because $B_n\in\calF$, and $\|P_n\|\leqslant1$ by definition of the norm.  If $N_n$ denotes the left shift on $\ell_1^n$, then $P_nL^kQ_n=N_n^k$ for every $k\geqslant0$, and consequently $P_nq_n(L)Q_n=q_n(N_n)$.

For the fixed polynomial $q_n$ from the preliminaries,
\[
       q_n(N_n)e_n=\sum_{k=0}^{n-1}\eps_ke_{n-k},
       \qquad
       \|q_n(N_n)\|_{\ell_1^n\to\ell_1^n}\geqslant n.
\]
Hence $\|q_n(L)\|_{\calS\to\calS}\geqslant n$.  Since $\|L\|\leqslant1$,
we have $V(L)\subseteq\overline\D$ and
$\sup_{V(L)}|q_n|\leqslant\Delta\sqrt n$.  As the admissible cardinalities
$n$ are unbounded,
\[
       \frac{\|q_n(L)\|}{\sup_{V(L)}|q_n|}
       \geqslant\Delta^{-1}\sqrt n
\]
is unbounded.  Therefore $\Psi(L,\calB(\calS))=\infty$.
\end{proof}

\begin{proof}[Proof of Corollary~\ref{cor:schreier-shift}]
The family $\calF_1$ is spreading, covers $\N$, and contains sets of arbitrarily large cardinality; for instance, $\{n,n+1,\ldots,2n-1\}\in\calF_1$.  The assertion is therefore an immediate consequence of Theorem~\ref{thm:combinatorial}.
\end{proof}

\section{Further shift consequences}\label{sec:further-shifts}

\subsection{Finite-dimensional, Calkin and subsymmetric tests}

We finish by recording several further forms of the shift obstruction.  The finite-dimensional version is often the most transparent.

\begin{proposition}\label{prop:finite-shift-test}
Let $X$ have a normalised $1$-subsymmetric basis $(e_j)_{j\geqslant1}$, put $E_n=[e_1,\ldots,e_n]$, and let $J_n\in\calB(E_n)$ be the nilpotent right shift $J_ne_j=e_{j+1}$ $(j<n)$ and $J_ne_n=0$.  If
\[
       \phi_X(n)=\left\|\sum_{j=1}^n e_j\right\|,
\]
then
\[
       \Psi(J_n,\calB(E_n))\geqslant \Delta^{-1}
       \max\left\{\frac{\phi_X(n)}{\sqrt n},\frac{\sqrt n}{\phi_X(n)}\right\}.
\]
\end{proposition}

\begin{proof}
The operator $J_n$ is a contraction: for scalars $a_j$,
\[
       \left\|\sum_{j=1}^{n-1}a_je_{j+1}\right\|
       =\left\|\sum_{j=1}^{n-1}a_je_j\right\|
       \leqslant \left\|\sum_{j=1}^{n}a_je_j\right\|,
\]
using subsymmetry and then suppression unconditionality.  Thus $V(J_n)\subseteq\overline\D$.  Use the fixed Rudin--Shapiro polynomial $q_n(z)=\sum_{k=0}^{n-1}\eps_kz^k$ from the preliminaries, for which $\|q_n\|_{\overline\D}\leqslant\Delta\sqrt n$.  By $1$-unconditionality,
\[
       \|q_n(J_n)e_1\|=\left\|\sum_{k=0}^{n-1}\eps_ke_{k+1}\right\|=\phi_X(n).
\]
Moreover, the coordinate functionals have norm one, and
$q_n(J_n)^*e_n^*=\sum_{k=0}^{n-1}\eps_ke_{n-k}^*$.  Evaluating this functional on $\sum_{k=0}^{n-1}\eps_ke_{n-k}$ gives $n$, while the latter vector has norm $\phi_X(n)$.  Hence $\|q_n(J_n)\|\geqslant\max\{\phi_X(n),n/\phi_X(n)\}$.  Dividing by $\sup_{V(J_n)}|q_n|\leqslant\Delta\sqrt n$ gives the assertion.
\end{proof}

The next proposition records the same obstruction in the Calkin algebra.  It is a Banach-space analogue of the fact that, for Hilbert-space contractions, passing to the Calkin algebra does not create a Crouzeix obstruction.

\begin{proposition}\label{prop:calkin-shift}
Let $1<p<\infty$, $p\neq2$, let $S$ be the unilateral right shift on $\ell_p$, and let $\pi:\calB(\ell_p)\to\calQ_p:=\calB(\ell_p)/\calK(\ell_p)$ be the quotient map.  Then
\[
       V(\pi(S),\calQ_p)=\overline\D
       \qquad\text{and}\qquad
       \Psi(\pi(S),\calQ_p)=\infty.
\]
\end{proposition}

\begin{proof}
Let $L$ denote the left shift.  Since $LS=I$ and $SL=I-P_1$, where $P_1$ has rank one, $\pi(S)$ is invertible with inverse $\pi(L)$.  Also $\|\pi(S)\|\leqslant1$ and $\|\pi(L)\|\leqslant1$.  Thus, if $\mu\in\sigma(\pi(S))$, then $|\mu|\leqslant1$ and $|\mu|^{-1}\leqslant\|\pi(S)^{-1}\|\leqslant1$, so $|\mu|=1$.  Hence $\sigma(\pi(S))\subseteq\T$.  Conversely, if $\lambda\in\T$, the vectors
\[
       x_N=N^{-1/p}\sum_{j=1}^N\lambda^{-(j-1)}e_j
\]
satisfy $\|(S-\lambda I)x_N\|\to0$.  The sequence $(x_N)$ is bounded and coordinatewise null, hence weakly null in the reflexive space $\ell_p$.  If $S-\lambda I$ were Fredholm, Atkinson's theorem \cite[Chapter~XI]{Conway} would give $B\in\calB(\ell_p)$ and compact $K$ with $B(S-\lambda I)=I-K$.  Compactness then gives $Kx_N\to0$, and hence $x_N=B(S-\lambda I)x_N+Kx_N\to0$, a contradiction.  Thus $\lambda\in\sigma(\pi(S))$.  Hence $\sigma(\pi(S))=\T$.  Since algebraic numerical ranges are convex and contain the spectrum, while $\|\pi(S)\|\leqslant1$, we have the complete chain
\[
       \overline\D=\conv\T
       \subseteq V(\pi(S),\calQ_p)
       \subseteq\|\pi(S)\|\overline\D
       \subseteq\overline\D.
\]
Therefore $V(\pi(S),\calQ_p)=\overline\D$.

Use the fixed Rudin--Shapiro polynomials $q_n(z)=\sum_{k=0}^{n-1}\eps_kz^k$ from the preliminaries.  We claim that the essential norm of $q_n(S)$ is at least $\max\{n^{1/p},n^{1/p'}\}$, where $p'$ is the conjugate exponent.  Indeed, $(e_m)$ and $(e_m^*)$ are weakly null in $\ell_p$ and $\ell_{p'}$, respectively, and compact operators send weakly null sequences to norm-null sequences.  Hence, for every compact $K$, one has $Ke_m\to0$ and $K^*e_m^*\to0$.  Applying $q_n(S)-K$ to $e_m$ and $(q_n(S)-K)^*$ to $e_{m+n-1}^*$ and letting $m\to\infty$ gives the two bounds $n^{1/p}$ and $n^{1/p'}$.  Therefore
\[
       \|q_n(\pi(S))\|=\|q_n(S)\|_{\rm e}
       \geqslant\max\{n^{1/p},n^{1/p'}\}.
\]
Since $V(\pi(S))=\overline\D$ and $\|q_n\|_{\overline\D}\leqslant\Delta\sqrt n$, the quotient defining $\Psi(\pi(S))$ is bounded below by $\Delta^{-1}n^{\max\{1/p,1/p'\}-1/2}$, which tends to infinity because $p\neq2$.
\end{proof}

We also record a general infinite-dimensional form.    If $X$ is a Banach space and $T\in\calB(X)$, let
\[
       W_X(T)=\{\langle Tx,x^*\rangle:\|x\|=\|x^*\|=1=\langle x,x^*\rangle\}
\]
be the spatial numerical range.  We shall use only the elementary inclusion
$W_X(T)\subseteq V(T,\calB(X))$.  The following proposition is in the spirit of the unilateral-shift argument in \cite[Proposition~4.15]{MT}; it gives a stronger algebraic-numerical-range conclusion.

\begin{proposition}\label{prop:subsym-shift}
Let $X$ have a normalised $1$-subsymmetric basis $(e_j)_{j\geqslant1}$, and let $S e_j=e_{j+1}$ be the right shift.  Put
\[
       \phi_X(n)=\left\|\sum_{j=1}^n e_j\right\|\qquad(n\geqslant1).
\]
If
\[
       \sup_{n\geqslant1}\max\left\{\frac{\phi_X(n)}{\sqrt n},\frac{\sqrt n}{\phi_X(n)}\right\}=\infty,
\]
then $\Psi(S,\calB(X))=\infty$.  Consequently, there is no constant $C$ such that
\[
       \|p(T)\|\leqslant C\sup_{z\in V(T,\calB(X))}|p(z)|\qquad(p\in\C[z],\ T\in\calB(X)),
\]
not even with $V(T,\calB(X))$ replaced by the spatial numerical range $W_X(T)$.
\end{proposition}

\begin{proof}
The shift is an isometry, so $V(S,\calB(X))\subseteq\overline\D$.  Use the fixed Rudin--Shapiro polynomials $q_n(z)=\sum_{k=0}^{n-1}\eps_kz^k$ from the preliminaries, with $\|q_n\|_{\overline\D}\leqslant\Delta\sqrt n$.  Since the basis is $1$-unconditional,
\[
       \|q_n(S)e_1\|=\left\|\sum_{k=0}^{n-1}\eps_ke_{k+1}\right\|=\phi_X(n).
\]
On the other hand, the coordinate functionals have norm one and
\[
       q_n(S)^*e_n^*=\sum_{k=0}^{n-1}\eps_ke_{n-k}^*,
\]
and its duality pairing with $\sum_{k=0}^{n-1}\eps_ke_{n-k}$ is $n$, while the latter vector has norm $\phi_X(n)$.  Hence
\[
       \|q_n(S)\|\geqslant \max\left\{\phi_X(n),\frac{n}{\phi_X(n)}\right\}.
\]
Therefore
\[
       \frac{\|q_n(S)\|}{\sup_{z\in V(S,\calB(X))}|q_n(z)|}
       \geqslant \Delta^{-1}\max\left\{\frac{\phi_X(n)}{\sqrt n},\frac{\sqrt n}{\phi_X(n)}\right\},
\]
which is unbounded by hypothesis.  This proves $\Psi(S,\calB(X))=\infty$.  Since $W_X(S)\subseteq V(S,\calB(X))$, the same polynomials also rule out the spatial-numerical-range estimate.
\end{proof}

We use the following non-degenerate convention.  An Orlicz function is a
convex non-decreasing function \(M:[0,\infty)\to[0,\infty)\) such that
\(M(0)=0\), \(M(t)>0\) for \(t>0\), and \(M(t)\to\infty\) as
\(t\to\infty\).  Its generalised inverse is
\[
       M^{-1}(s)=\inf\{t\geqslant0:M(t)\geqslant s\}.
\]

\begin{lemma}\label{lem:orlicz-inverse}
Let \(M\) be an Orlicz function and let \(\ell_M\) carry the Luxemburg norm.
Then
\[
       \left\|\sum_{j=1}^n e_j\right\|_{\ell_M}
       =\frac1{M^{-1}(1/n)}.
\]
Moreover, the following conditions are equivalent:
\begin{enumerate}
\item[(i)] \(M^{-1}(1/n)\asymp n^{-1/2}\) as \(n\to\infty\);
\item[(ii)] \(M^{-1}(s)\asymp s^{1/2}\) as \(s\downarrow0\);
\item[(iii)] \(M(t)\asymp t^2\) as \(t\downarrow0\).
\end{enumerate}
\end{lemma}

\begin{proof}
Being finite-valued and convex, $M$ is continuous on $(0,\infty)$; moreover,
convexity and $M(0)=0$ give $M(t)\leqslant tM(1)$ for $0\leqslant t\leqslant1$,
so $M$ is continuous on $[0,\infty)$.  Convexity and the condition $M(t)>0$
for $t>0$ also imply that $M$ is strictly increasing: if $0<s<t$, then
$M(s)\leqslant(s/t)M(t)<M(t)$.  Thus $M^{-1}$ is the ordinary continuous
inverse, although the generalised-inverse notation remains convenient.  For \(x_n=\sum_{j=1}^n e_j\), the defining inequality for the Luxemburg norm is
\[
       nM(1/\lambda)\leqslant1.
\]
The least admissible \(\lambda\) is therefore \(1/M^{-1}(1/n)\), which proves the first assertion.

The implication \textup{(ii)}\(\Rightarrow\)\textup{(i)} is immediate.  Conversely,
choose \(n\) so that \(1/(n+1)<s\leqslant1/n\).  Monotonicity gives
\[
       M^{-1}(1/(n+1))\leqslant M^{-1}(s)
       \leqslant M^{-1}(1/n),
\]
and
\((n+1)^{-1/2}\asymp s^{1/2}\asymp n^{-1/2}\); hence
\textup{(i)} implies \textup{(ii)}.  Finally, suppose that
\[
       c s^{1/2}\leqslant M^{-1}(s)\leqslant C s^{1/2}
\]
for small \(s\).  Taking \(s=M(t)\) and using
\(M^{-1}(M(t))=t\) gives
\[
       C^{-2}t^2\leqslant M(t)\leqslant c^{-2}t^2
\]
for small \(t\).  Conversely, if
\(a t^2\leqslant M(t)\leqslant b t^2\) near zero, put
\(t=M^{-1}(s)\) to obtain
\[
       b^{-1/2}s^{1/2}\leqslant M^{-1}(s)
       \leqslant a^{-1/2}s^{1/2}.
\]
This proves the equivalence of \textup{(ii)} and \textup{(iii)}.
\end{proof}

\begin{corollary}\label{cor:orlicz}
Let \(M\) be an Orlicz function and let \(\ell_M\) be the corresponding
Orlicz sequence space, endowed with the Luxemburg norm
\[
       \|x\|_M=
       \inf\left\{\lambda>0:
       \sum_{j=1}^{\infty}M\!\left(\frac{|x_j|}{\lambda}\right)
       \leqslant1\right\};
\]
see \cite{Musielak} for the standard background.  If \(M\) is not equivalent
to \(t^2\) at zero, then the right shift \(S\) on \(\ell_M\) satisfies
\(\Psi(S,\calB(\ell_M))=\infty\).  In particular, no Crouzeix-type estimate
with a constant independent of \(T\in\calB(\ell_M)\) can hold, either for the
algebraic or for the spatial numerical range.
\end{corollary}

\begin{proof}
We give the direct shift test, so no density of \(c_{00}\) in \(\ell_M\) is
required.  The right shift is an isometry and therefore
\(V(S)\subseteq\overline\D\).  Put
\[
       \phi_M(n)=\left\|\sum_{j=1}^n e_j\right\|_{\ell_M},
       \qquad
       a_1=\phi_M(1)>0.
\]
Lemma~\ref{lem:orlicz-inverse} gives
\[
       \phi_M(n)=\frac1{M^{-1}(1/n)}.
\]
For the fixed Rudin--Shapiro polynomial
\(q_n(z)=\sum_{k=0}^{n-1}\eps_kz^k\), symmetry of the Luxemburg norm gives
\[
       \|q_n(S)\|
       \geqslant\frac{\phi_M(n)}{a_1}.
\]
The norm of every coordinate functional is exactly \(1/a_1\).  Indeed, the
coordinate projection is contractive, so
\[
       |x_j|a_1=\|x_je_j\|_M\leqslant\|x\|_M,
\]
which gives \(\|e_j^*\|\leqslant1/a_1\), while equality follows by evaluating
\(e_j^*\) at \(e_j/a_1\).  Since
\[
       q_n(S)^*e_n^*=\sum_{k=0}^{n-1}\eps_ke_{n-k}^*,
\]
its duality pairing with \(\sum_{k=0}^{n-1}\eps_ke_{n-k}\) is \(n\), and the
latter vector has norm \(\phi_M(n)\).  Consequently,
\[
       \|q_n(S)\|\geqslant a_1\frac{n}{\phi_M(n)}.
\]
It follows that
\[
       \frac{\|q_n(S)\|}{\sup_{V(S)}|q_n|}
       \geqslant \Delta^{-1}
       \max\left\{
       \frac{\phi_M(n)}{a_1\sqrt n},
       \frac{a_1\sqrt n}{\phi_M(n)}
       \right\}.
\]
If this expression were bounded, then
\(M^{-1}(1/n)\asymp n^{-1/2}\).  Lemma~\ref{lem:orlicz-inverse} would imply
\(M(t)\asymp t^2\) near zero, contrary to the hypothesis.  Hence the displayed
quotient is unbounded and \(\Psi(S,\calB(\ell_M))=\infty\).  Since the spatial
numerical range is contained in the algebraic numerical range, the same
polynomials also rule out the spatial estimate.
\end{proof}

\begin{remark}
The same criterion applies to any normalised $1$-symmetric basis and, more generally, to any normalised $1$-subsymmetric basis.  It provides a common sufficient obstruction whenever the fundamental function is not of Hilbertian order $\sqrt n$ up to constants.  It is not a characterisation: when $\phi_X(n)\asymp\sqrt n$, the shift test is silent, including for sequence spaces whose fundamental function is Hilbertian but whose norm is not induced by an inner product.  Additional multiplier information would be needed to treat that borderline class.
\end{remark}

\appendix
\section{Quantitative and bounded-order complements}\label{sec:appendix-quantitative}

The results in this appendix are not needed for the qualitative fixed-degree
theorem.  They record the explicit cubic bookkeeping and a direct estimate for
root-of-unity relations.

\subsection{The cubic constant}

The recursive bound in Theorem~\ref{thm:all-degrees} is explicit but extremely large.  In degree three the three possible root-multiplicity patterns can be tracked much more efficiently.

\begin{theorem}\label{thm:degree-three}
For the degree-three constant one has
\[
       5\leqslant\Gamma_3<184\,267.
\]
\end{theorem}

\begin{proof}
The lower bound is Proposition~\ref{prop:gamma-n-lower}.  Put
\[
       C_2=\sqrt{1+(2\mathrm e-1)^2},
       \qquad
       D=C_2+1.
\]
The constant \(D\) has a simple bookkeeping meaning.  On each invariant
quadratic restriction Theorem~\ref{thm:degree-two} costs at most \(C_2K\),
and isolating the relevant interpolation coefficient requires subtracting one
scalar value of modulus at most \(K\).  Thus \(D=C_2+1\) is the natural
one-step loss when the degree-two calculus is inserted into the cubic proof;
no optimisation of this loss is intended.  Put
\[
       \eta=\frac{\log2}{5},
       \qquad
       \gamma=\frac{\eta}{4},
       \qquad
       \delta=\frac{\eta\gamma}{2}
       =\frac{(\log2)^2}{200}.
\]
Let \(M_*=16\,600\).  We work with an operator \(T\), obtained from the original algebra element by left multiplication and an affine change of variables, and put
\[
       K=\sup_{z\in V(T)}|p(z)|>0,
       \qquad
       M=\|T\|.
\]
Elements of degree at most two are already controlled by \(C_2\), so only the three cubic multiplicity patterns need consideration.

\smallskip
\noindent\emph{Triple root.}
After a translation, \(T^3=0\).  The proof of Proposition~\ref{prop:nilpotents}, with \(m=3\), gives
\[
       \delta M\overline\D\subseteq V(T).
\]
Since
\[
       p(T)=p(0)I+p'(0)T+\frac{p''(0)}2T^2,
\]
Cauchy's estimates on this disc yield
\[
       \|p(T)\|
       \leqslant
       \left(1+\delta^{-1}+\delta^{-2}\right)K.
\]

\smallskip
\noindent\emph{One double root and one simple root.}
After an affine change,
\[
       T^2(T-I)=0.
\]
Set
\[
       P=T^2,
       \qquad
       N=T-T^2.
\]
Then \(P^2=P\), \(N^2=0\), and \(PN=NP=0\), while
\begin{equation}\label{eq:cubic-double-functional}
       p(T)=p(0)I+\bigl(p(1)-p(0)\bigr)P+p'(0)N.
\end{equation}
Assume first that \(M\leqslant M_*\).  To estimate the projection term, choose \(x\) with \(\|x\|=1\) and \(\mu=\|Px\|>0\), put
\[
       v=Tx,
       \qquad
       y=\frac{Px}{\mu},
\]
and let \(E=T|_{[v,y]}\).  Since \(T^3=T^2\),
\[
       Tv=Px=\mu y,
       \qquad
       Ty=y.
\]
Thus \([v,y]\) is invariant and \(E^2=E\).  Consequently,
\[
       (p(E)-p(0)I)v=(p(1)-p(0))\mu y.
\]
Lemma~\ref{lem:invariant-restriction} and Theorem~\ref{thm:degree-two} give
\(\|p(E)\|\leqslant C_2K\).  Since \(|p(0)|\leqslant K\) and
\(\|v\|\leqslant M\),
\[
       |p(1)-p(0)|\mu
       \leqslant (\|p(E)\|+|p(0)|)\|v\|
       \leqslant DKM.
\]
Letting \(\mu\uparrow\|P\|\) gives
\begin{equation}\label{eq:cubic-P-bound}
       |p(1)-p(0)|\,\|P\|\leqslant DKM.
\end{equation}
If \(N\neq0\), choose \(x\) almost norming \(N\), put
\[
       v=(I-T)x,
       \qquad
       y=\frac{Nx}{\|Nx\|},
\]
and restrict \(T\) to \([v,y]\).  Here
\[
       Tv=Nx=\|Nx\|y,
       \qquad
       Ty=0,
\]
because \(TN=T^2-T^3=0\).  Thus the restriction \(E\) is square-zero and
\[
       (p(E)-p(0)I)v=p'(0)\|Nx\|y.
\]
Since \(\|v\|\leqslant M+1\), the same degree-two estimate gives
\begin{equation}\label{eq:cubic-N-bound}
       |p'(0)|\,\|N\|\leqslant DK(M+1).
\end{equation}
Equations \eqref{eq:cubic-double-functional}--\eqref{eq:cubic-N-bound} imply
\begin{equation}\label{eq:cubic-bounded-common}
       \|p(T)\|\leqslant\bigl[1+D(2M+1)\bigr]K.
\end{equation}

Suppose now that \(M>M_*\).  Since \(B_3=8\gamma^{-2}<6700\), Lemma~\ref{lem:general-large-disc} gives \(\varrho\overline\D\subseteq V(T)\), where \(\varrho=\delta M>1\).  Let
\[
       h(z)=p(0)+p'(0)z+p[0,0,1]z^2
\]
be the Hermite remainder of \(p\) modulo \(z^2(z-1)\).  Then \(p-h\) is
divisible by \(z^2(z-1)\), so \(h(T)=p(T)\).  On \(|z|=\varrho\),
formula~\eqref{eq:cauchy-divided-difference} gives
\[
\begin{split}
       |p'(0)|
       &\leqslant\frac{1}{2\pi}(2\pi\varrho)
       \frac{K}{\varrho^2}=\frac K\varrho,\\
       |p[0,0,1]|
       &\leqslant\frac{1}{2\pi}(2\pi\varrho)
       \frac{K}{\varrho^2(\varrho-1)}
       =\frac{K}{\varrho(\varrho-1)}.
\end{split}
\]
Consequently,
\begin{equation}\label{eq:cubic-double-large}
       \frac{\|p(T)\|}{K}
       \leqslant
       1+\frac1\delta+
       \frac{M}{\delta(\delta M-1)}.
\end{equation}

\smallskip
\noindent\emph{Three distinct roots.}
Choose a pair of roots of maximal distance and normalise them to \(0\) and \(1\).  The third root \(\lambda\) then satisfies
\[
       |\lambda|\leqslant1,
       \qquad
       |1-\lambda|\leqslant1,
\]
and
\[
       T(T-I)(T-\lambda I)=0.
\]
Let \(P_0,P_1,P_\lambda\) be the spectral projections.  We use
\begin{equation}\label{eq:cubic-distinct-functional}
       p(T)=p(\lambda)I+
       \bigl(p(0)-p(\lambda)\bigr)P_0+
       \bigl(p(1)-p(\lambda)\bigr)P_1.
\end{equation}
Assume first that \(M\leqslant M_*\).  Choose \(x\) almost norming \(P_0\), put \(\mu=\|P_0x\|\),
\[
       v=(T-I)x,
       \qquad
       y=\frac{P_0x}{\mu}.
\]
Since
\[
       (T-\lambda I)(T-I)=\lambda P_0,
\]
we have
\[
       Ty=0,
       \qquad
       Tv=\lambda v+\lambda\mu y.
\]
Thus \([v,y]\) is invariant, \(v+\mu y\) is a \(\lambda\)-eigenvector, and the
restriction \(E=T|_{[v,y]}\) has degree at most two.  Decomposing
\(v=(v+\mu y)-\mu y\) into its \(\lambda\)- and \(0\)-eigencomponents gives
\[
       \bigl(p(E)-p(\lambda)I\bigr)v
       =-\bigl(p(0)-p(\lambda)\bigr)\mu y.
\]
Lemma~\ref{lem:invariant-restriction} and Theorem~\ref{thm:degree-two} therefore yield
\begin{equation}\label{eq:cubic-P0-bound}
       |p(0)-p(\lambda)|\,\|P_0\|
       \leqslant DK(M+1).
\end{equation}
For \(P_1\), choose \(x\) almost norming \(P_1\), put
\(\mu=\|P_1x\|\), \(y=P_1x/\mu\), and \(v=Tx\), and let
\(E=T|_{[v,y]}\).  The identity
\[
       T(T-\lambda I)=(1-\lambda)P_1
\]
gives
\[
       Ty=y,
       \qquad
       Tv=\lambda v+(1-\lambda)\mu y.
\]
Hence \(v-\mu y\) is a \(\lambda\)-eigenvector and
\[
       \bigl(p(E)-p(\lambda)I\bigr)v
       =\bigl(p(1)-p(\lambda)\bigr)\mu y.
\]
Applying the invariant-restriction and degree-two estimates as above, and
then letting \(\mu\uparrow\|P_1\|\), yields
\begin{equation}\label{eq:cubic-P1-bound}
       |p(1)-p(\lambda)|\,\|P_1\|
       \leqslant DKM.
\end{equation}
Together with \eqref{eq:cubic-distinct-functional}, these inequalities again give \eqref{eq:cubic-bounded-common}.

If \(M>M_*\), Lemma~\ref{lem:general-large-disc} gives \(\varrho\overline\D\subseteq V(T)\), where \(\varrho=\delta M\).  The Newton interpolant
\[
       h(z)=p(0)+p[0,\lambda]z+p[0,\lambda,1]z(z-\lambda)
\]
matches \(p\) at \(0,\lambda,1\); hence \(p-h\) is divisible by
\(z(z-\lambda)(z-1)\) and \(h(T)=p(T)\).  On \(|z|=\varrho\),
formula~\eqref{eq:cauchy-divided-difference} gives
\[
\begin{split}
       |p[0,\lambda]|
       &\leqslant\frac{1}{2\pi}(2\pi\varrho)
       \frac K{\varrho(\varrho-1)}
       =\frac K{\varrho-1},\\
       |p[0,\lambda,1]|
       &\leqslant\frac{1}{2\pi}(2\pi\varrho)
       \frac K{\varrho(\varrho-1)^2}
       =\frac K{(\varrho-1)^2}.
\end{split}
\]
Since \(\|T(T-\lambda I)\|\leqslant M(M+1)\),
\begin{equation}\label{eq:cubic-distinct-large}
       \frac{\|p(T)\|}{K}
       \leqslant
       1+\frac{M}{\delta M-1}
       +\frac{M(M+1)}{(\delta M-1)^2}.
\end{equation}

It remains only to record the numerical bounds.  We have \(D<5.55\) and \(\delta>0.0024\).  Hence, for \(M\leqslant M_*\),
\[
       1+D(2M+1)<184\,267.
\]
The triple-root estimate is smaller than \(174\,029\).  The relevant
large-norm expressions are decreasing.  Indeed,
\[
 \frac{d}{dM}\frac{M}{\delta(\delta M-1)}
 =-\frac{1}{\delta(\delta M-1)^2}<0,
\]
and
\[
 \frac{d}{dM}\left(
       \frac{M}{\delta M-1}
       +\frac{M(M+1)}{(\delta M-1)^2}
       \right)
 =-\frac{1}{(\delta M-1)^2}
  -\frac{(\delta+2)M+1}{(\delta M-1)^3}<0
\]
whenever \(\delta M>1\).  Evaluating the right-hand sides of
\eqref{eq:cubic-double-large} and \eqref{eq:cubic-distinct-large} at
\(M=M_*\), respectively, and using \(\delta>0.0024\), gives
\[
\begin{aligned}
  1+\frac1\delta+
  \frac{M_*}{\delta(\delta M_*-1)}
  &<178\,499,\\
  1+\frac{M_*}{\delta M_*-1}
  +\frac{M_*(M_*+1)}{(\delta M_*-1)^2}
  &<183\,106.
\end{aligned}
\]
Thus \eqref{eq:cubic-double-large} gives the first bound in the case of
one double and one simple root, whereas
\eqref{eq:cubic-distinct-large} gives the second bound in the case of
three distinct roots.  Consequently every quotient in the definition of
\(\Gamma_3\) is strictly smaller than \(184\,267\), as required.
\end{proof}

\subsection{Roots of unity}

The cyclic analogue of the ratio-drop argument gives a more explicit estimate for every fixed root-of-unity relation.

\begin{proposition}\label{prop:roots}
For every integer $m\geqslant2$ there is a constant $D_m$ such that
\[
       \Psi(U,\calA)\leqslant D_m
\]
whenever $U^m=1$ in a unital Banach algebra $\calA$.
\end{proposition}

\begin{proof}
We again work in $\calB(X)$.  Put $M=\|U\|$ and fix
\[
       \eta=\frac{\log2}{5},\qquad \gamma=\frac{\eta}{4},\qquad
       \delta_m=\frac{\eta}{2}\gamma^{m-2},
\]
and set
\[
       M_m=\max\left\{\gamma^{-(m-1)/2},\frac{2}{\delta_m}\right\}.
\]
We first prove the disc inclusion for $M>M_m$.  Choose $\eps>0$ so small that $(1-\eps)M>\gamma^{-(m-1)/2}$.  Pick $x_0\in X$ with $\|x_0\|=1$ and $\|Ux_0\|>(1-\eps)M$, put $x_j=U^jx_0$, and write
$R_j=\|x_{j+1}\|/\|x_j\|$ for $0\leqslant j\leqslant m-1$, where $x_m=x_0$.  Then $R_0\cdots R_{m-1}=1$.  If $R_{j+1}>\gamma R_j$ for all $0\leqslant j\leqslant m-2$, then
\[
       1=R_0\cdots R_{m-1}>R_0^m\gamma^{m(m-1)/2}>1,
\]
a contradiction.  Hence there is a first $k\leqslant m-2$ with $R_{k+1}\leqslant\gamma R_k$, and then $R_k\geqslant\gamma^{m-2}(1-\eps)M$.

Set $v=x_k/\|x_k\|$ and $w=x_{k+1}/\|x_{k+1}\|$.  Then $Uv=R_kw$ and $Uw=R_{k+1}u$ for some unit vector $u$.  For $\zeta\in\T$ let $f(r)=\|v+r\zeta w\|$.  On $[2,7]$ we have $f(r)\geqslant r-1>0$, $f(2)\leqslant3$, and $f(7)\geqslant6$.  Since $f$ is convex,
\[
       \int_2^7\frac{f'_+(s)}{f(s)}\,ds
       =\log f(7)-\log f(2)\geqslant\log2,
\]
so for some $r\in[2,7]$ one has $f'_+(r)/f(r)\geqslant\eta$.  Also $r/f(r)\leqslant2$.  Since
\[
       (I+t\zeta U)(v+r\zeta w)=v+(r+tR_k)\zeta w+tr\zeta^2R_{k+1}u,
\]
the triangle inequality gives
\[
       \|I+t\zeta U\|\geqslant
       \frac{f(r+tR_k)-trR_{k+1}}{f(r)}.
\]
Formula \eqref{eq:support} therefore gives
\[
       \sup_{z\in V(U)}\operatorname{Re}(\zeta z)
       \geqslant R_k\frac{f'_+(r)}{f(r)}-R_{k+1}\frac{r}{f(r)}
       \geqslant \eta R_k-2R_{k+1}
       \geqslant \frac{\eta}{2}R_k.
\]
Letting $\eps\downarrow0$ yields $\delta_mM\overline\D\subseteq V(U)$ whenever $M>M_m$.  For these $M$ we also have $\varrho=\delta_mM\geqslant2$.

Put $K=\sup_{V(U)}|p|>0$.  If $p(z)=\sum c_jz^j$, Cauchy's estimate on $\varrho\overline\D$ gives $|c_j|\leqslant K\varrho^{-j}$.  Since $U^m=1$,
\[
       p(U)=\sum_{\ell=0}^{m-1}d_\ell U^\ell,
       \qquad d_\ell=\sum_{s\geqslant0}c_{\ell+sm},
\]
where the sums are finite.  Hence
\[
       \|p(U)\|\leqslant \frac{K}{1-2^{-m}}
       \sum_{\ell=0}^{m-1}\delta_m^{-\ell}.
\]
This proves the desired estimate for $M>M_m$.

It remains to treat $M\leqslant M_m$.  Let $\Lambda_m=\{\lambda\in\C:\lambda^m=1\}$ and, for $\lambda\in\Lambda_m$, put
\[
       P_\lambda=\frac{1}{m}\sum_{j=0}^{m-1}\lambda^{-j}U^j .
\]
These are the spectral projections of $U$ corresponding to the roots in $\Lambda_m$; some of them may be zero.  If $M\leqslant M_m$, then
$\|P_\lambda\|\leqslant m^{-1}\sum_{j=0}^{m-1}M_m^j$.  The identity
$p(U)=\sum_{\lambda\in\Lambda_m}p(\lambda)P_\lambda$ follows from the discrete Fourier formula for the remainder modulo $z^m-1$.  Whenever $P_\lambda\neq0$, we have $(U-\lambda I)P_\lambda=0$, so $U-\lambda I$ is not invertible and $\lambda\in\sigma(U)\subseteq V(U)$.  Therefore $|p(\lambda)|\leqslant K$ for all non-zero projections.  Hence
\[
       \|p(U)\|\leqslant K\sum_{j=0}^{m-1}M_m^j .
\]
Combining the bounded and large cases gives a constant $D_m$ depending only on $m$.
\end{proof}

\begin{remark}
Theorem~\ref{thm:all-degrees} settles the fixed-degree problem qualitatively, but the constants produced by the spectral-splitting induction grow very rapidly.  Propositions~\ref{prop:nilpotents} and~\ref{prop:roots} remain useful because they give direct and substantially more explicit bounds for two important classes of relations.  No optimisation of the general constants is attempted here.
\end{remark}

\end{document}